\documentclass[11pt,reqno]{amsart}
 \usepackage{amssymb,amsfonts,amscd,amsbsy, color, amsmath}
\usepackage[mathscr]{eucal}
\usepackage{url}
\usepackage{graphicx}
\usepackage{mathtools}

\newtheorem{theorem}{Theorem}[section]
\newtheorem{lemma}[theorem]{Lemma}

\newtheorem{corollary}[theorem]{Corollary}
\theoremstyle{definition}

\numberwithin{equation}{section}

\newcommand{\sgn}{\operatorname{sgn}}

\makeatletter
\def\imod#1{\allowbreak\mkern5mu({\operator@font mod}\,\,#1)}
\makeatother
\allowdisplaybreaks

\makeatletter
\let\@@pmod\pmod
\DeclareRobustCommand{\pmod}{\@ifstar\@pmods\@@pmod}
\def\@pmods#1{\mkern4mu({\operator@font mod}\mkern 6mu#1)}
\makeatother

\begin{document}

\title[The colored Jones polynomial for double twist knots, II]{The colored Jones polynomial and Kontsevich-Zagier series for double twist knots, II}

\author{Jeremy Lovejoy}

\author{Robert Osburn}

\address{Current Address: Department of Mathematics, University of California, Berkeley, 970 Evans Hall \#3780,
Berkeley, CA 94720-3840, USA }

\address{Permanent Address: CNRS, Universit{\'e} Denis Diderot - Paris 7, Case 7014, 75205 Paris Cedex 13, FRANCE}

\address{School of Mathematics and Statistics, University College Dublin, Belfield, Dublin 4, Ireland}

\address{Max-Planck-Institut f{\"u}r Mathematik, Vivatsgasse 7, D-53111, Bonn, Germany}

\email{lovejoy@math.cnrs.fr}

\email{robert.osburn@ucd.ie}


\date{\today}

\begin{abstract}
Let $K_{(m,p)}$ denote the family of double twist knots where $2m-1$ and $2p$ are non-zero integers denoting the number of half-twists in each region.  Using a result of Takata, we prove a formula for the colored Jones polynomial of $K_{(-m,-p)}$ and $K_{(-m,p)}$. The latter case leads to new families of $q$-hypergeometric series generalizing the Kontsevich-Zagier series.    These generalized Kontsevich-Zagier series are conjectured to be quantum modular forms.   We also use Bailey pairs and formulas of Walsh to find Habiro-type expansions for the colored Jones polynomials of $K_{(m,p)}$ and $K_{(m,-p)}$.    
\end{abstract}

\maketitle

\section{Introduction}

Let $K$ be a knot and $J_N(K;q)$ be the usual $N$th colored Jones polynomial, normalized to be 1 for the unknot. Formulas for $J_N(K;q)$ in terms of $q$-hypergeometric series have been proved for several families of knots \cite{habiro, hikami1, hikami2, thang, masbaum, walsh}; these  have played a prominent role in numerous studies in quantum topology and modular forms \cite{bhl, gukov, hikami3, hikami4, hl1, hl2, zagier}.  In recent work \cite{loCJP}, the authors used a theorem of Takata \cite{takata} to find $q$-hypergeometric expressions for the colored Jones polynomial of double twist knots where each of the two regions consisted of an even number of half-twists.   This led to a doubly infinite family of $q$-series generalizing the famous Kontsevich-Zagier series \cite{z-1,zagier},
\begin{equation}  \label{K-Zseries}
F(q) = \sum_{n \geq 0} (1-q)(1-q^2) \cdots (1-q^n).
\end{equation}
These generalized Kontsevich-Zagier series are conjectured to be new families of quantum modular forms.   Comparing with previously known expressions for the colored Jones polynomials of double twist knots due to Lauridsen \cite{La} led to generalizations of a $q$-series ``identity'' involving $F(q)$ due to Bryson, Ono, Pitman, and Rhoades \cite{BOPR} -- namely, for any root of unity $q$ one has
\begin{equation}
F(q^{-1}) = \sum_{n \geq 0} q^{n+1}(1-q)^2 \cdots (1-q^n)^2. 
\end{equation}
For a complete description of these results, see \cite{loCJP}.

Here we turn our attention to double twist knots where one region has an odd number of half-twists.    Recall the standard $q$-hypergeometric notation
\begin{equation*}
(a)_n = (a;q)_n := \prod_{k=0}^{n-1} (1-aq^{k})
\end{equation*}
and the usual $q$-binomial coefficient

\begin{equation}
\begin{bmatrix} n \\ k \end{bmatrix} = \begin{bmatrix} n \\ k \end{bmatrix}_{q} := \frac{(q)_n}{(q)_{n-k} (q)_k}.
\end{equation}
Consider the family of double twist knots $K_{(m,p)}$ where $2m-1$ and $2p$ are non-zero integers denoting the number of half-twists in each respective region of Figure \ref{fig:dt}. Positive integers correspond to right-handed half-twists and negative integers correspond to left-handed half-twists.

\begin{figure}[ht]
\includegraphics[width=7.5cm, height=5.0cm]{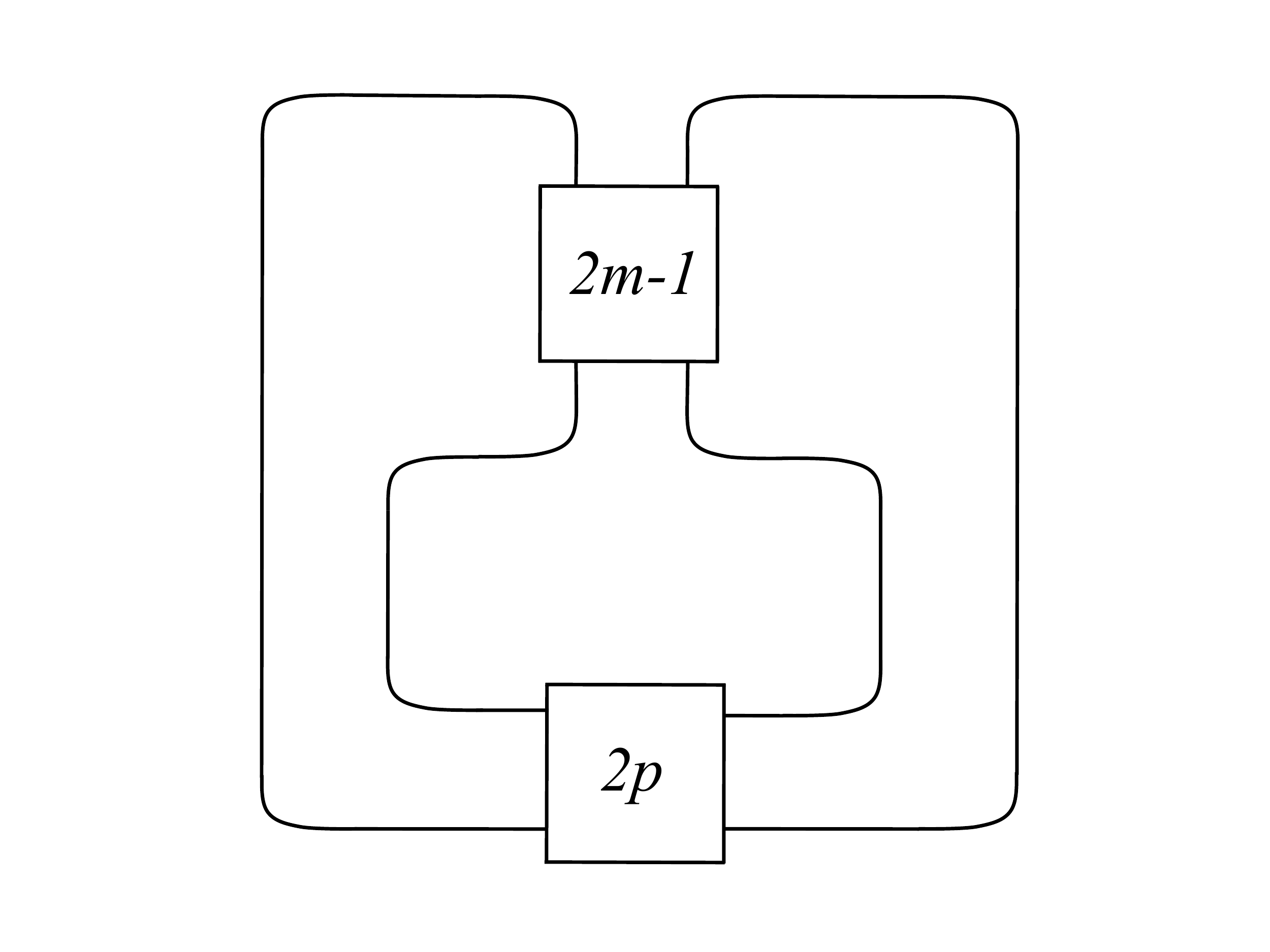}
\caption{Double twist knots}
\label{fig:dt}
\end{figure}

\noindent To state the case $K_{(-m,-p)}$, we define the functions $\epsilon_{i,j,m}$ and $\gamma_{i,m}$ by
\begin{equation} \label{epsilondef}
\epsilon_{i,j,m} = 
\begin{cases}
1, &\text{if $j \equiv -i$ or $-i-1$ $\pmod{2m+1}$}, \\
-1, &\text{if $j \equiv i$ or $i-1$ $\pmod{2m+1}$}, \\
0, &\text{otherwise}
\end{cases}
\end{equation}
where $1 \leq i < j \leq (2m+1)p-1$ with $(2m +1)\nmid i$ and $j \not \equiv m\pmod{2m+1}$ and
\begin{equation} \label{gammadef}
\gamma_{i,m} =
\begin{cases}
1, &\text{if $i \equiv 1,\dots, m-1 \pmod{2m+1}$}, \\
-1 &\text{if $i \equiv 0, m+1, \dots , 2m \pmod{2m+1}$}, \\
0 &\text{if $i \equiv m \pmod{2m+1}$}
\end{cases}
\end{equation}
where $1 \leq i \leq (2m+1)p-2$. Our first main result is the following.

\begin{theorem} \label{main1}
For positive integers $m$ and $p$, we have
\begin{align}  \label{JNK-m-p}
J&_{N}(K_{(-m,-p)}; q) \nonumber \\ 
&= q^{(p-1)(N-1)}\sum_{N-1 \geq n_{(2m+1)p-1} \geq \cdots \geq n_1 \geq 0} (q^{1-N})_{n_{(2m+1)p-1}} (-1)^{n_{(2m+1)p-1}} q^{-\binom{n_{(2m+1)p-1} + 1}{2}} \nonumber \\
&  \qquad \qquad \qquad \qquad \times\prod_{\substack{1 \leq i < j \leq (2m+1)p-1 \\ (2m +1) \nmid i \\ j \not \equiv m \pmod*{2m+1}}} q^{\epsilon_{i,j,m}n_in_j} 
 \prod_{\substack{i=1 \\ i \equiv m,\, 2m+1 \pmod*{2m+1}}}^{(2m+1)p-2} (-1)^{n_{i}} q^{N n_{i} + \binom{n_{i} + 1}{2}} \nonumber \\
 & \qquad \qquad \qquad \qquad \times \prod_{i=1}^{(2m+1)p - 2} q^{-n_i n_{i+1} + \gamma_{i, m} n_i}\begin{bmatrix} n_{i+1} \\ n_{i} \end{bmatrix}. 
\end{align}
\end{theorem}

\noindent For an example of Theorem \ref{main1}, take $m=3$ and $p=1$. We then have

\begin{align*}
J_N&(K_{(-3,-1)};q) 
= \sum_{N-1 \geq n_6 \geq n_5 \geq n_4 \geq n_3 \geq n_2 \geq n_1 \geq 0}  (q^{1-N})_{n_6} (-1)^{n_3 + n_6} q^{Nn_3 + \binom{n_3 + 1}{2} - \binom{n_6+1}{2}} \\
& \qquad \qquad \qquad \qquad \qquad \qquad \times q^{n_1(n_5 + n_6) + n_2(n_4 + n_5) - n_1n_2 - n_2n_3 - n_4n_5 - n_5n_6} \\
&\qquad \qquad \qquad \qquad  \qquad \qquad \qquad \times q^{n_1 + n_2 - n_4 - n_5} \begin{bmatrix} n_6 \\ n_5 \end{bmatrix} \begin{bmatrix} n_5 \\ n_4 \end{bmatrix} \begin{bmatrix} n_4 \\ n_3 \end{bmatrix} \begin{bmatrix} n_3 \\ n_2 \end{bmatrix} \begin{bmatrix} n_2 \\ n_1 \end{bmatrix}.
 \end{align*}

For the case of $K_{(-m,p)}$, define the functions $\Delta_{i,j,m}$ and $\beta_{i,m}$ by
\begin{equation} \label{Deltadef}
\Delta_{i,j,m} = 
\begin{cases}
1, &\text{if $j \equiv -i$ or $-i+1$ $\pmod{2m+1}$}, \\
-1, &\text{if $j \equiv i$ or $i+1$ $\pmod{2m+1}$}, \\
0, &\text{otherwise}
\end{cases}
\end{equation}
where $1 \leq i < j \leq (2m+1)p$ with $(2m+1) \nmid i$ and $j \not \equiv m+1 \pmod{2m+1}$ and
\begin{equation} \label{betadef}
\beta_{i,m} =
\begin{cases}
1, &\text{if $i \equiv 1,\dots, m \pmod{2m+1}$}, \\
-1 &\text{if $i \equiv m+1,\dots, 2m \pmod{2m+1}$}, \\
0, &\text{if $i \equiv 0 \pmod{2m+1}$} 
\end{cases}
\end{equation}
where $1 \leq i \leq (2m+1)p-1$. For convenience, we define $\beta_{i,0}=0$ for $1 \leq i \leq p-1$. Our second main result is the following.

\begin{theorem} \label{main2}
For a nonnegative integer $m$ and positive integer $p$, we have
\begin{align} \label{JNK-mp}
J_{N}&(K_{(-m,p)}; q) \nonumber \\ 
&= q^{p(1-N)} \sum_{N-1 \geq n_{(2m+1)p} \geq \cdots \geq n_1 \geq 0} (q^{1-N})_{n_{(2m+1)p}}(-1)^{n_{(2m+1)p}} q^{-\binom{n_{(2m+1)p} + 1}{2}} \nonumber \\
&  \qquad \qquad \qquad \qquad \times \prod_{\substack{1 \leq i < j \leq (2m+1)p \\ (2m+1) \nmid i \\ j \not \equiv m+1 \pmod*{2m+1}}} q^{\Delta_{i,j,m} n_i n_j} \prod_{\substack{i=1 \\ i \equiv m+1,\, 2m+1 \pmod*{2m+1}}}^{(2m+1)p-1} (-1)^{n_{i}} q^{-N n_{i} + \binom{n_{i} + 1}{2}} \nonumber \\
& \qquad \qquad \qquad \qquad \times \prod_{i=1}^{(2m+1)p - 1} q^{\beta_{i,m} n_i}\begin{bmatrix} n_{i+1} \\ n_{i}  \end{bmatrix}.  
\end{align}
\end{theorem}

The case $m=0$ of Theorem \ref{main2} was proved by Hikami \cite{hikami1}. Here $K_{(0, p)}=T_{(2,2p+1)}$, the family of right-handed torus knots. Thus, one recovers $J_{N}(T_{(2,2p+1)}; q)$ by taking $m=0$ in (\ref{JNK-mp}). To see this, we first rewrite (\ref{JNK-mp}) as

\begin{align} \label{JNK-mp-whatif}
J_{N}&(K_{(-m,p)}; q) \nonumber \\ 
&= q^{p(1-N)} \sum_{N-1 \geq n_{(2m+1)p} \geq \cdots \geq n_1 \geq 0} (q^{1-N})_{n_{(2m+1)p}}q^{Nn_{(2m+1)p}} q^{-2\binom{n_{(2m+1)p} + 1}{2}} \nonumber \\
&  \qquad \qquad \qquad \qquad \times \prod_{\substack{1 \leq i < j \leq (2m+1)p \\ (2m+1) \nmid i \\ j \not \equiv m+1 \pmod*{2m+1}}} q^{\Delta_{i,j,m} n_i n_j} \prod_{\substack{i=1 \\ i \equiv m+1 \pmod*{2m+1}}}^{(2m+1)p} (-1)^{n_{i}} q^{-N n_{i} + \binom{n_{i} + 1}{2}} \nonumber \\
& \qquad \qquad \qquad \qquad \times \prod_{\substack{i=1 \\ i \equiv 2m+1 \pmod*{2m+1}}}^{(2m+1)p} (-1)^{n_{i}} q^{-N n_{i} + \binom{n_{i} + 1}{2}}  \prod_{i=1}^{(2m+1)p - 1} q^{\beta_{i,m} n_i}\begin{bmatrix} n_{i+1} \\ n_{i} \end{bmatrix}.
\end{align}

\noindent For $m=0$, the first product in (\ref{JNK-mp-whatif}) is empty while the second and third products in (\ref{JNK-mp-whatif}) are equal.   Taking $\beta_{i,0} = 0$ in \eqref{betadef}, we have (cf. Proposition 9 in \cite{hikami1})

\begin{equation}  \label{t22p+1}
J_{N}(T_{(2,2p+1)}; q) = q^{p(1-N)} \sum_{N-1 \geq n_{p} \geq \cdots \geq n_1 \geq 0} (q^{1-N})_{n_{p}}  q^{-Nn_{p}} \prod_{i=1}^{p-1} q^{n_i(n_i + 1 -2N)} \begin{bmatrix} n_{i+1} \\ n_{i} \end{bmatrix}. 
\end{equation}

\noindent For another example of Theorem \ref{main2}, consider $m=p=2$. We then have

\begin{align*}
J_N&(K_{(-2,2)};q) \\
&= q^{2(1-N)} \sum_{N-1 \geq n_{10} \geq n_9 \geq n_8 \geq n_7 \geq n_6 \geq n_5 \geq n_4 \geq n_3 \geq n_2 \geq n_1 \geq 0}  (q^{1-N})_{n_{10}} (-1)^{n_3 + n_5 + n_8 + n_{10}} q^{-N(n_3 + n_5+ n_8)} \\
&\qquad \qquad \quad \times q^{-\binom{n_{10}+1}{2} + \binom{n_8+1}{2} + \binom{n_5+1}{2} + \binom{n_3+1}{2} +n_1(-n_2 + n_4 + n_5 - n_6 - n_7 + n_9 + n_{10}) + n_2(n_4 - n_7 + n_9)} \\
& \qquad \qquad \quad \hspace{.05in} \times q^{n_3(-n_4 + n_7 - n_9) + n_4(-n_5 + n_6 + n_7 - n_9 - n_{10}) + n_6(-n_7 + n_9 + n_{10}) + n_7n_9 - n_8n_9 - n_9 n_{10}} \\ 
& \qquad \qquad \quad \hspace{.1in} \times q^{n_1 + n_2 - n_3 - n_4 + n_6 + n_7 - n_8 - n_9}\begin{bmatrix} n_{10} \\ n_9 \end{bmatrix} \begin{bmatrix} n_9 \\ n_8 \end{bmatrix} \begin{bmatrix} n_8 \\ n_7 \end{bmatrix} \begin{bmatrix} n_7 \\ n_6\end{bmatrix} \begin{bmatrix} n_6 \\ n_5 \end{bmatrix}  \begin{bmatrix} n_5 \\ n_4 \end{bmatrix} \begin{bmatrix} n_4 \\ n_3 \end{bmatrix} \begin{bmatrix} n_3 \\ n_2 \end{bmatrix} \begin{bmatrix} n_2 \\ n_1 \end{bmatrix}.
\end{align*}

Recall that
\begin{equation} \label{mirror}
J_{N}(K; q^{-1}) = J_{N}(K^{*}; q),
\end{equation}
where $K^*$ denotes the mirror image of the knot $K$.   Thus, since $K_{(-m,-p)}$ is the mirror image of $K_{(m+1,p)}$ and $K_{(0,p-1)}$ is the mirror image of $K_{(0,-p)}$, equations \eqref{JNK-m-p} and \eqref{JNK-mp} cover all of the double twist knots in this family, up to a substitution of $q$ by $q^{-1}$. Combined with Theorems 1.1 and 1.2 in \cite{loCJP}, we have $q$-hypergeometric series expressions of this type for all double twist knots.

Another type of $q$-hypergeometric formula for the colored Jones polynomial can be deduced from formulas of Walsh \cite{walsh} together with the theory of Bailey pairs. These formulas are our third main result.

\begin{theorem} \label{main3}
For positive integers $m$ and $p$, we have
\begin{align} \label{cyclotomicmp}
J_N&(K_{(m,p)}; q)  \nonumber \\
&= q^{p(1-N^2)}\sum_{\substack{n \geq 0 \\ n =n_m \geq n_{m-1} \geq \cdots \geq n_1 \geq 0 \\ n = s_p \geq s_{p-1} \geq \cdots \geq s_1 \geq 0}}
\frac{(q^{1+N})_{n} (q^{1-N})_n}{(q)_{n_1}} q^n \prod_{i=1}^{m-1} q^{n_i^2+n_i} \begin{bmatrix} n_{i+1} \\ n_i \end{bmatrix} \prod_{j=1}^{p-1} q^{s_j^2+s_j} \begin{bmatrix} s_{j+1} \\ s_j \end{bmatrix}
\end{align}
and 

\begin{align}  \label{cyclotomicm-p}
J_N&(K_{(m,-p)}; q)  \nonumber \\
&= q^{-p(1-N^2)}\sum_{\substack{n \geq 0 \\ n = n_m \geq n_{m-1} \geq \cdots \geq n_1 \geq 0 \\ n = s_p \geq s_{p-1} \geq \cdots \geq s_1 \geq 0}} 
\frac{(q^{1+N})_{n} (q^{1-N})_n}{(q)_{n_1}} (-1)^nq^{-\binom{n+1}{2}} \nonumber \\
& \qquad \qquad \qquad \qquad \qquad \qquad \qquad \qquad \times \prod_{i=1}^{m-1} q^{n_i^2+n_i} \begin{bmatrix} n_{i+1} \\ n_i \end{bmatrix} \prod_{j=1}^{p-1} q^{-s_j - s_{j+1}s_j} \begin{bmatrix} s_{j+1} \\ s_j \end{bmatrix}.
\end{align}
\end{theorem}

In view of \eqref{JNK-mp} and \eqref{cyclotomicm-p}, we define the $q$-series $\mathbb{F}_{m,p}(q)$ for $m \geq 0$ and $p \geq 1$ and $\mathbb{U}_{m,p}(x;q)$ for $m,p \geq 1$ by

\begin{align} 
\mathbb{F}_{m,p}&(q) \nonumber \\
& = q^p \sum_{n_{(2m+1)p} \geq \cdots \geq n_1 \geq 0} (q)_{n_{(2m+1)p}} (-1)^{n_{(2m+1)p}} q^{-\binom{n_{(2m+1)p} + 1}{2}} \prod_{\substack{1 \leq i < j \leq (2m+1)p \\ (2m+1) \nmid i \\ j \not \equiv m+1 \pmod*{2m+1}}} q^{\Delta_{i,j,m} n_i n_j} \nonumber \\
& \qquad \qquad \qquad \qquad \times \prod_{\substack{i=1 \\ i \equiv m+1,\, 2m+1 \pmod*{2m+1}}}^{(2m+1)p-1} (-1)^{n_{i}} q^{\binom{n_{i} + 1}{2}} \prod_{i=1}^{(2m+1)p - 1} q^{\beta_{i,m} n_i}\begin{bmatrix} n_{i+1} \\ n_{i}  \end{bmatrix} 
\end{align}
and
\begin{align}
\mathbb{U}_{m,p}(x;q) \nonumber \\
& = q^{-p} \sum_{\substack{n \geq 0 \\ n = n_m \geq n_{m-1} \geq \cdots \geq n_1 \geq 0 \\ n = s_p \geq s_{p-1} \geq \cdots \geq s_1 \geq 0}} 
\frac{(-xq)_n (-x^{-1} q)_n}{(q)_{n_1}} (-1)^nq^{-\binom{n+1}{2}} \nonumber \\
& \qquad \qquad \qquad \qquad \qquad \qquad \qquad \qquad \times \prod_{i=1}^{m-1} q^{n_i^2+n_i} \begin{bmatrix} n_{i+1} \\ n_i \end{bmatrix} \prod_{j=1}^{p-1} q^{-s_j - s_{j+1}s_j} \begin{bmatrix} s_{j+1} \\ s_j \end{bmatrix}.
\end{align}
Note that neither $\mathbb{F}_{m,p}(q)$ nor $\mathbb{U}_{m,p}(x;q)$ is defined anywhere except at roots of unity. In this case, we have
\begin{equation} \label{Fatroot}
\mathbb{F}_{m,p}(\zeta_N) = J_N(K_{(-m,p)};\zeta_N)
\end{equation}
and 
\begin{equation} \label{Uatroot}
\mathbb{U}_{m,p}(-1; \zeta_N) = J_N(K_{(m,-p)};\zeta_N) 
\end{equation}  
for any $N$th root of unity $\zeta_N$. By (\ref{mirror}), (\ref{Fatroot}) and (\ref{Uatroot}) and since the mirror image of $K_{(-m,p)}$ is $K_{(m+1,-p)}$, we immediately have the following.

\begin{corollary}
If $\zeta_N$ is any root $N$th root of unity, then we have 
\begin{equation} \label{FUdual}
\mathbb{F}_{m,p}(\zeta_N) = \mathbb{U}_{m+1,p}(-1; \zeta_N^{-1}).
\end{equation}
\end{corollary}

\noindent Similar ``dualities" involving $q$-hypergeometric series at roots of unity can be found in \cite{BOPR,Co1,folsometal,hl1,loCJP}.   As the case $\mathbb{F}_{0,1}(q)$ is equal to $q$ times the Kontsevich-Zagier series \eqref{K-Zseries}, we refer to the $q$-series $\mathbb{F}_{m,p}(q)$ as the \emph{Kontsevich-Zagier series for odd double twist knots}. 

Similarly, motivated by \eqref{JNK-m-p} and \eqref{cyclotomicmp}, we define the $q$-series $\mathfrak{F}_{m,p}(q)$ and $\mathfrak{U}_{m,p}(x;q)$ for $m,p \geq 1$ by

\begin{align}
\mathfrak{F}&_{m,p}(q) \nonumber \\ 
& = q^{1-p} \sum_{n_{(2m+1)p-1} \geq \cdots \geq n_1 \geq 0} (q)_{n_{(2m+1)p-1}} (-1)^{n_{(2m+1)p-1}} q^{-\binom{n_{(2m+1)p-1} + 1}{2}} \prod_{\substack{1 \leq i < j \leq (2m+1)p-1 \\ (2m +1) \nmid i \\ j \not \equiv m \pmod*{2m+1}}} q^{\epsilon_{i,j,m}n_in_j} \nonumber \\
& \qquad \qquad \qquad \qquad \times \prod_{\substack{i=1 \\ i \equiv m,\, 2m+1 \pmod*{2m+1}}}^{(2m+1)p-2} (-1)^{n_{i}} q^{\binom{n_{i} + 1}{2}} \prod_{i=1}^{(2m+1)p - 2} q^{-n_i n_{i+1} + \gamma_{i, m} n_i}\begin{bmatrix} n_{i+1} \\ n_{i} \end{bmatrix} 
\end{align}
and
\begin{equation}
\mathfrak{U}_{m,p}(x;q)= q^{p}\sum_{\substack{n \geq 0 \\ n =n_m \geq n_{m-1} \geq \cdots \geq n_1 \geq 0 \\ n = s_p \geq s_{p-1} \geq \cdots \geq s_1 \geq 0}}
\frac{(-xq)_n (-x^{-1} q)_n}{(q)_{n_1}} q^n \prod_{i=1}^{m-1} q^{n_i^2+n_i} \begin{bmatrix} n_{i+1} \\ n_i \end{bmatrix} \prod_{j=1}^{p-1} q^{s_j^2+s_j} \begin{bmatrix} s_{j+1} \\ s_j \end{bmatrix}.
\end{equation}

\noindent Here, $\mathfrak{U}_{m,p}(x;q)$ is well-defined for $|q| < 1$ and for $q$ a root of unity when $x=-1$ while $\mathfrak{F}_{m,p}(q)$ is only defined at roots of unity. Then 
\begin{equation}
\mathfrak{F}_{m,p}(\zeta_N) = J_{N}(K_{(-m,-p)}; \zeta_N)
\end{equation}
and 
\begin{equation}
\mathfrak{U}_{m,p}(-1; \zeta_N) = J_{N}(K_{(-m,p)}; \zeta_N)
\end{equation}
for any $N$th root of unity $\zeta_N$, giving the following.

\begin{corollary}
If $\zeta_N$ is any root $N$th root of unity, then we have
\begin{equation} \label{FUdual2}
\mathfrak{F}_{m,p}(\zeta_N) = \mathfrak{U}_{m+1,p}(-1; \zeta_N^{-1}).
\end{equation}
\end{corollary}

The rest of this paper is organized as follows. In Section 2, we recall Takata's main theorem and provide some preliminaries. In Sections 3 and 4, we prove Theorems \ref{main1} and \ref{main2}.    In Section 5, we prove Theorem \ref{main3}.    In Section 6, we conclude with some remarks.

\section{Preliminaries}

We begin by recalling the setup from \cite{takata}. Let $l$ and $t$ be coprime odd integers with $l > t \geq 1$ and $p^{\prime} := \frac{l-1}{2}$. For $1 \leq j \leq p^{\prime}$, define integers $r(j)$ such that $r(j) \equiv (2j-1)t \pmod{2l}$ and $-l < r(j) < l$. We put $\sigma_j := (-1)^{\lfloor \frac{(2j-1)t}{l}\rfloor}$, $r^{\prime}(j) := \frac{\mid r(j) \mid + 1}{2}$ and $i_{r^{\prime}(j)}=j$ (and thus $i_k = j$ if and only if $r^{\prime}(j)=k$). For an integer $i$, sgn($i$) denotes the sign of $i$. Let $\underline{n} = (n_1, \dotsc, n_{p^{\prime}})$ and $n_s=0$ for $s \leq 0$. Finally, define

\begin{equation} \label{kappa}
\kappa(p^{\prime})= \left\{
\begin{array}{ll}
-Nn_{p^{\prime}} & \text{if $\sigma_{p^{\prime}}=-1$,} \\
0 & \text{if $\sigma_{p^{\prime}}=1$} \\
\end{array} \right. \\
\end{equation}

\noindent and

\begin{equation} \label{tau}
\tau(j) = \left\{
\begin{array}{ll}
(-1)^{n_j - n_{j-1}} & \text{if $\sigma_j=-1$,} \\
q^{\binom{n_j - n_{j-1} + 1}{2}} & \text{if $\sigma_{j}=1$.} \\
\end{array} \right. \\
\end{equation}

Consider the family of 2-bridge knots $\mathfrak{b}(l,t)$ (see \cite{bz} or \cite{mpvl}). The main result in \cite{takata} is an explicit formula for the colored Jones polynomial of $\mathfrak{b}(l,t)^{*}$.

\begin{theorem} \label{takata} We have
\begin{equation} \label{spt}
J_N(\mathfrak{b}(l,t)^{*}; q) = \sum_{N-1 \geq n_{p^{\prime}} \geq \dotsc \geq n_1 \geq 0} q^{a({\underline{n}})N + b_1(\underline{n}) + b_2(\underline{n})} X(\underline{n})
\end{equation}
where\footnote{Note that there is a misprint in the definition of $X(\underline{n})$ in \cite{takata}.   Each $\overline{q}$ in the prefactor should be $q$.}

\begin{align*}
a({\underline{n}}) & = -\frac{1}{2} \sum_{j=1}^{p^{\prime}} \Biggl( \sum_{k=r^{\prime}(j)}^{p^{\prime}} (\sigma_{i_k} + \sigma_{i_{p^{\prime} + 1 -k}})\Biggr) (n_j - n_{j-1}) - \frac{1}{2} \sum_{j=1}^{p^{\prime} - 1} (\sigma_{j+1} + \sigma_{p^{\prime} + 1 - j}) n_j  \\
& - \frac{1}{2} (\sigma_{p^{\prime}} + 1) n_{p^{\prime}} - \sum_{j=1}^{p^{\prime}} \sigma_j, \\
b_1(\underline{n}) & = - a({\underline{n}}) + \sum_{k=1}^{\frac{l-t}{2}} \frac{1-\sigma_{i_k}}{2} n_{i_k - 1} - \sum_{k=\frac{l-t}{2}+1}^{p^{\prime}} n_{i_k - 1} + \sum_{k=\frac{l-t}{2}+1}^{p^{\prime}} \frac{1+\sigma_{i_k}}{2} n_{i_k} - (1+\sigma_{p^{\prime}}) n_{p^{\prime}}  \\
& + \frac{1}{2} \sum_{j=1}^{p^{\prime} - 1} (\sigma_{j+1} - \sigma_j) n_j  \\
& - \frac{1}{2} \sum_{k=1}^{p^{\prime} - 1} \sum_{k^{\prime} = k+1}^{p^{\prime}} \frac{1 + \sgn(i_k - i_{k^{\prime}})}{2}(\sigma_{i_k} - \sigma_{i_{k^{\prime}}})(n_{i_k} - n_{i_k - 1})(n_{i_{k^{\prime}}} - n_{i_{k^{\prime}} - 1})  \\
& + \sum_{j=1}^{p^{\prime}} \sigma_j \Bigl( \sum_{k=1}^{r^{\prime}(j)} (n_{i_k} - n_{i_k - 1}) \Bigr) n_{j-1},  \\
b_2(\underline{n}) & = \left\{
\begin{array}{ll}
\displaystyle \sum_{k=\frac{l-t}{2} + 1}^{\frac{t-1}{2}} \frac{1+ \sigma_{i_k}}{2} n_{i_k - 1} & \text{if $l < 2t$,} \\
\displaystyle - \sum_{k=\frac{t+1}{2} + 1}^{\frac{l-t}{2}} \frac{1+ \sigma_{i_k}}{2} n_{i_k - 1} & \text{if $l > 2t,$} \\
\end{array} \right.  \\
X(\underline{n}) & = (-1)^{n_{p^{\prime}}} q^{\kappa(p^{\prime})} \frac{(q)_{N-1} (q)_{n_{p^{\prime}}}}{(q)_{N-n_{p^{\prime}} - 1}} \prod_{j=1}^{p^{\prime}} \frac{\tau(j)}{(q)_{n_j - n_{j-1}}}. 
\end{align*}
\end{theorem}

Our interest will be to apply Theorem \ref{takata} to the case of the double twist knots $K_{(m+1,p)}=\mathfrak{b}(4mp + 2p -1, 4mp - 1)$ and $K_{(m+1,-p)}=\mathfrak{b}(4mp + 2p + 1, 4mp + 1)$, whose mirror images are $K_{(-m,-p)}$ and $K_{(-m,p)}$, respectively (cf. \cite{Tran}). In order to facilitate these computations, we need the following results concerning $\sigma_j$, $i_k$ and $\sigma_{i_k}$. We omit the proofs as they are straightforward generalizations of Lemmas 6--9 in \cite{takata}.

\begin{lemma} \label{l10} For $l=4mp + 2p - 1$ and $t=4mp - 1$, we have \\

\begin{enumerate}
\item[(i)] $\sigma_{j} = \left\{
\begin{array}{ll}
1 & \text{if $j \equiv 1, 2, \dotsc, m \pmod{2m+1}$,} \\
-1 & \text{if $j \equiv 0, m+1, \dotsc, 2m \pmod{2m+1}.$} \\
\end{array} \right.$ \\

\item[(ii)] To compute $i_k$, apply the following algorithm. Divide the integers from $1$ to $p'$ into $2m$ intervals, each of length $p$, and a final interval of length $p -1$. The value of $i_k$ is $(2m+1)(k-1) + m  + 1$ in the first interval and $(2m+1)(2p - k) + m$ in the second.  If $j>1$ is odd, then to obtain the value of $i_k$ in the $j$th interval, subtract $2(2m+1)p - 1$ from the formula for $i_k$ in the $(j-2)$th interval. If $j > 2$ is even, then to obtain the value of $i_k$ in the $j$th interval, add $2(2m+1)p - 1$ to the formula for $i_k$ in the $(j-2)$th interval. \\

\item[(iii)] To compute $\sigma_{i_k}$, apply the following algorithm.  Divide the integers from $1$ to $p'$ into $2m$ intervals, each of length $p$, and a final interval of length $p - 1$. The value of $\sigma_{i_k}$ alternates between $-1$ and $1$ starting with $-1$ in the first interval. 
\end{enumerate}
\end{lemma}

\begin{lemma} \label{l11} Let $l=4mp + 2p - 1$ and $t=4mp - 1$. Then for $1\leq k \leq p'$ and $1\leq j \leq p'-1$ we have \\

\begin{enumerate}

\item[(i)] $\sigma_{i_k} + \sigma_{i_{p^{\prime} + 1 - k}} = \left\{
\begin{array}{ll}
2 & \text{if $ip + 1 \leq k \leq (i+1)p - 1$ for $i=1, 3, \dotsc, 2m-1$,} \\
-2 & \text{if $ip + 1 \leq k \leq (i+1)p - 1$ for $i=0,2, \dotsc, 2m$},\\
0 & \text{if $k=ip$ for $i=1, 2, \dotsc, 2m.$} \\
\end{array} \right.$ \\

\item[(ii)] $\sigma_{j+1} + \sigma_{p^{\prime} + 1 - j} =  \left\{
\begin{array}{ll}
-2 & \text{if $j \equiv m \pmod{2m+1}$,} \\
0 & \text{otherwise.} \\
\end{array} \right.$ \\
\end{enumerate}
\end{lemma}

\begin{lemma} \label{l12} For $l=4mp + 2p + 1$ and $t=4mp + 1$, we have \\

\begin{enumerate}
\item[(i)] $\sigma_{j} = \left\{
\begin{array}{ll}
1 & \text{if $j \equiv 1, 2, \dotsc, m+1 \pmod{2m+1}$,} \\
-1 & \text{if $j \equiv 0, m+2, \dotsc, 2m \pmod{2m+1}.$} \\
\end{array} \right.$ \\

\item[(ii)] To compute $i_k$, apply the following algorithm. Divide the integers from $1$ to $p'$ into $2m+1$ intervals, each of length $p$. The value of $i_k$ is $(2m+1)(k-1) + m  + 1$ in the first interval and $(2m+1)(2p - k) + m + 2$ in the second.  If $j>1$ is odd, then to obtain the value of $i_k$ in the $j$th interval, subtract $2(2m+1)p + 1$ from the formula for $i_k$ in the $(j-2)$th interval. If $j > 2$ is even, then to obtain the value of $i_k$ in the $j$th interval, add $2(2m+1)p + 1$ to the formula for $i_k$ in the $(j-2)$th interval. \\

\item[(iii)] To compute $\sigma_{i_k}$, apply the following algorithm.  Divide the integers from $1$ to $p'$ into $2m+1$ intervals, each of length $p$. The value  of $\sigma_{i_k}$ alternates between $1$ and $-1$ starting with $1$ in the first interval. 
\end{enumerate}
\end{lemma}

\begin{lemma} \label{l13} Let $l=4mp + 2p + 1$ and $t=4mp + 1$. Then for $1\leq k \leq p'$ and $1\leq j \leq p'-1$ we have \\

\begin{enumerate}

\item[(i)] $\sigma_{i_k} + \sigma_{i_{p^{\prime} + 1 - k}} = \left\{
\begin{array}{ll}
2 & \text{if $ip + 1 \leq k \leq (i+1)p$ for $i=0, 2, \dotsc, 2m$,} \\
-2 & \text{if $ip + 1 \leq k \leq (i+1)p$ for $i=1,3, \dotsc, 2m-1.$}\\
\end{array} \right.$ \\

\item[(ii)] $\sigma_{j+1} + \sigma_{p^{\prime} + 1 - j} =  \left\{
\begin{array}{ll}
2 & \text{if $j \equiv 0 \pmod{2m+1}$,} \\
0 & \text{otherwise.} \\
\end{array} \right.$ \\
\end{enumerate}
\end{lemma}

We now illustrate the computation of $a({\underline{n}})$ and $b_1({\underline{n}}) + b_2({\underline{n}})$ for $l=10p+1$ and $t=8p+1$. The routine evaluation of $X(\underline{n})$ is left to the reader. First, we take $m=2$ in Lemmas \ref{l12} and \ref{l13} to obtain

\begin{equation} \label{step0}
\sigma_{j} = \left\{
\begin{array}{ll}
1 & \text{if $j \equiv 1, 2, 3 \pmod{5}$,} \\
-1 & \text{if $j \equiv 0, 4 \pmod{5},$} \\
\end{array} \right. \\
\end{equation}

\begin{equation} \label{step1}
i_k = \left\{
\begin{array}{ll}
5(k-1) + 3 & \text{if $1 \leq k \leq p$,} \\
5(2p-k) + 4 & \text{if $p+1 \leq k \leq 2p$,} \\
5k-10p-3 & \text{if $2p+1 \leq k \leq 3p$,} \\
 20p-5k+5 & \text{if $3p+1 \leq k \leq 4p$,} \\
5k-20p-4 & \text{if $4p+1 \leq k \leq 5p$,} \\
\end{array} \right. \\
\end{equation}

\begin{equation} \label{step2}
\sigma_{i_k} = \left\{
\begin{array}{ll}
1 & \text{if $1 \leq k \leq p$,} \\
-1 & \text{if $p+1 \leq k \leq 2p$,} \\
1 & \text{if $2p+1 \leq k \leq 3p$,} \\
-1 & \text{if $3p+1 \leq k \leq 4p$,} \\
1 & \text{if $4p+1 \leq k \leq 5p$,} \\
\end{array} \right. \\
\end{equation}

\begin{equation} \label{step3}
\sigma_{i_k} + \sigma_{i_{5p+1-k}} = \left\{
\begin{array}{ll}
2 & \text{if $1 \leq k \leq p$,} \\
-2 & \text{if $p+1 \leq k \leq 2p$,} \\
2 & \text{if $2p+1 \leq k \leq 3p$,} \\
-2 & \text{if $3p+1 \leq k \leq 4p$,} \\
2 & \text{if $4p+1 \leq k \leq 5p$,} \\
\end{array} \right. \\
\end{equation}
and
\begin{equation} \label{step4}
\sigma_{j+1} + \sigma_{5p+1 - j} = \left\{
\begin{array}{ll}
2 & \text{if $j \equiv 0 \pmod{5}$,} \\
0 & \text{otherwise.}
\end{array} \right. \\
\end{equation}

\noindent Applying (\ref{step0}), (\ref{step1}), (\ref{step3}), (\ref{step4}), reindexing and after considerable simplification, we obtain that $a({\underline{n}})$ equals

\begin{align} \label{a}
& -\frac{1}{2} \sum_{j=1}^{5p} \Biggl( \sum_{k=r^{\prime}(j)}^{5p} (\sigma_{i_k} + \sigma_{i_{5p+1- k}})\Biggr) (n_j - n_{j-1}) - \sum_{j=1}^{p-1} n_{5j} - p \nonumber \\
& = -\frac{1}{2} \Biggl[ \sum_{j=1}^{p}  \Biggl( \sum_{k=4p-j+1}^{5p} (\sigma_{i_k} + \sigma_{i_{5p+1 -k}})\Biggr) (n_{5j} - n_{5j-1}) \nonumber \\
& + \sum_{j=1}^{p} \Biggl( \sum_{k=4p+j}^{5p} (\sigma_{i_k} + \sigma_{i_{5p+1 -k}})\Biggr) (n_{5j-4} - n_{5j-5}) + \sum_{j=1}^{p} \Biggl( \sum_{k=2p+j}^{5p} (\sigma_{i_k} + \sigma_{i_{5p+1 -k}})\Biggr) (n_{5j-3} - n_{5j-4}) \nonumber \\
& + \sum_{j=1}^{p} \Biggl( \sum_{k=j}^{5p} (\sigma_{i_k} + \sigma_{i_{5p+1 -k}})\Biggr) (n_{5j-2} - n_{5j-3}) + \sum_{j=1}^{p} \Biggl( \sum_{k=2p-j+1}^{5p} (\sigma_{i_k} + \sigma_{i_{5p+1 -k}})\Biggr) (n_{5j-1} - n_{5j-2}) \Biggr] \nonumber \\
& - \sum_{j=1}^{p-1} n_{5j} - p \nonumber \\
& = \sum_{j=1}^{p} n_{5j-2} - \sum_{j=1}^{p-1} n_{5j} - p.
\end{align}

\noindent By (\ref{step0}) and (\ref{step2}), the second and fifth sums in $b_1({\underline{n}})$ are zero. We then use (\ref{step0})--(\ref{step2}) and reindex to obtain

\begin{align}  \label{thirdpart}
-\sum_{k=p+1}^{5p} n_{i_{k} - 1} & = - \Biggl( \sum_{k=p+1}^{2p} n_{i_{k} - 1} + \sum_{k=2p+1}^{3p} n_{i_{k} - 1} + \sum_{k=3p+1}^{4p} n_{i_{k} - 1} + \sum_{k=4p+1}^{5p} n_{i_{k} - 1} \Biggr) \nonumber \\
& = - \Biggl( \sum_{j=1}^{p} (n_{5j-2} + n_{5j-4} + n_{5j-1} + n_{5j-5}) \Biggr),
\end{align}

\begin{equation} \label{fourthpart}
\sum_{k=p+1}^{5p} \frac{1+\sigma_{i_k}}{2} n_{i_k} = \sum_{k=2p+1}^{3p} n_{i_k} + \sum_{k=4p+1}^{5p} n_{i_k} = \sum_{j=1}^{p} (n_{5j-3} + n_{5j-4}),
\end{equation}

\begin{equation} \label{sixthpart}
\frac{1}{2} \sum_{j=1}^{5p-1} (\sigma_{j+1} - \sigma_j) n_j = \sum_{j=1}^{p-1} n_{5j} - \sum_{j=1}^{p} n_{5j-2}
\end{equation}
and
\begin{equation} \label{b2}
b_2({\underline{n}}) = \sum_{k=p+1}^{4p} \frac{1+\sigma_{i_k}}{2} n_{{i_k} - 1} = \sum_{k=2p+1}^{3p} n_{{i_k} - 1} = \sum_{j=1}^{p} n_{5j-4}. 
\end{equation}

\noindent By (\ref{a})--(\ref{b2}), the sum of $b_2({\underline{n}})$ and the first six terms in $b_1({\underline{n}})$ equals

\begin{equation} \label{sixb1b2}
p + \sum_{j=1}^{p-1} n_{5j}  + \sum_{j=1}^{p} (n_{5j-4} + n_{5j-3} - n_{5j-2} - n_{5j-1}).
\end{equation}

\noindent To compute the seventh term in $b_1({\underline{n}})$, we use (\ref{step1}) and (\ref{step2}) to observe that $k < k^{\prime}$ and $\sigma_{i_k} \neq \sigma_{i_{k^{\prime}}}$ if and only if either $1 \leq k \leq p$ and $p+1 \leq k^{\prime} \leq 2p$ or $1 \leq k \leq p$ and $3p+1 \leq k^{\prime} \leq 4p$ or $p+1 \leq k \leq 2p$ and $2p+1 \leq k^{\prime} \leq 3p$ or $p+1 \leq k \leq 2p$ and $4p+1 \leq k^{\prime} \leq 5p$ or $2p+1 \leq k \leq 3p$ and $3p+1 \leq k^{\prime} \leq 4p$ or $3p+1 \leq k \leq 4p$ and $4p+1 \leq k^{\prime} \leq 5p$. Also, $\sgn(i_k - i_{k^{\prime}}) = 1$ if and only if $i_k > i_{k^{\prime}}$ and either $i_k=5k-2$ for $1 \leq k \leq p$ and $i_{k^{\prime}} = 10p - 5k^{\prime} + 4$ for $p_1 + 1 \leq k^{\prime} \leq 2p$ or $i_k=5k-2$ for $1 \leq k \leq p$ and $i_{k^{\prime}} = 20p - 5k + 5$ for $3p+1 \leq k \leq 4p$ or $i_k = 20p - 5k + 4$ for $p+1 \leq k \leq 2p$ or $i_{k^{\prime}} = 5k^{\prime} - 10p - 3$ for $2p+1 \leq k^{\prime} \leq 3p$ or $i_k = 20p - 5k + 4$ for $p+1 \leq k \leq 2p$ and $i_{k^{\prime}} = 5k^{\prime} - 20p - 4$ for $4p+1 \leq k \leq 5p$ or $i_k = 5k-10p-3$ for $2p+1 \leq k \leq 3p$ and $i_{k^{\prime}} = 20p - 5k^{\prime} + 5$ for $3p+1 \leq k \leq 4p$ or $i_k = 20p - 5k + 5$ for $3p+1 \leq k \leq 4p$ and $i_{k^{\prime}}  = 5k^{\prime} - 20p - 4$ for $4p+1 \leq k \leq 5p$. Taking these cases into account and reindexing, we have

\begin{align} \label{sevenb1}
& - \frac{1}{2} \sum_{k=1}^{5p-1} \sum_{k^{\prime} = k+1}^{5p} \frac{1 + \sgn(i_k - i_{k^{\prime}})}{2}(\sigma_{i_k} - \sigma_{i_{k^{\prime}}})(n_{i_k} - n_{i_k - 1})(n_{i_{k^{\prime}}} - n_{i_{k^{\prime}} - 1}) \nonumber \\
& = - \sum_{k=1}^{p} \sum_{k^{\prime} = 2p-k+2}^{2p} (n_{i_k} - n_{i_k - 1})(n_{i_{k^{\prime}}} - n_{i_{k^{\prime}} - 1}) - \sum_{k=1}^{p} \sum_{k^{\prime}=4p-k+2}^{4p} (n_{i_k} - n_{i_k - 1})(n_{i_{k^{\prime}}} - n_{i_{k^{\prime}} - 1}) \nonumber \\
& + \sum_{k=p+1}^{2p} \sum_{k^{\prime} = 2p+1}^{4p-k+1} (n_{i_k} - n_{i_k - 1})(n_{i_{k^{\prime}}} - n_{i_{k^{\prime}} - 1}) 
+ \sum_{k=p+1}^{2p} \sum_{k^{\prime} = 4p+1}^{6p-k+1} (n_{i_k} - n_{i_k - 1})(n_{i_{k^{\prime}}} - n_{i_{k^{\prime}} - 1}) \nonumber \\
& - \sum_{k=2p+1}^{3p} \sum_{k^{\prime} = 6p-k+2}^{4p} (n_{i_k} - n_{i_k - 1})(n_{i_{k^{\prime}}} - n_{i_{k^{\prime}} - 1}) + \sum_{k=3p+1}^{4p} \sum_{k^{\prime} = 4p+1}^{8p-k+1} (n_{i_k} - n_{i_k - 1})(n_{i_{k^{\prime}}} - n_{i_{k^{\prime}} - 1}) \nonumber \\
& = - \sum_{j=1}^{p} \sum_{j^{\prime}=1}^{j} (n_{5j-2} - n_{5j-3})(n_{5j^{\prime} - 6} - n_{5j^{\prime} - 7}) - \sum_{j=1}^{p} \sum_{j^{\prime}=1}^{j} (n_{5j-2} - n_{5j-3})(n_{5j^{\prime} -5} - n_{5j^{\prime} - 6}) \nonumber \\
& + \sum_{j=1}^{p} \sum_{j^{\prime} = 1}^{j} (n_{5j-1} - n_{5j-2})(n_{5j^{\prime} - 3} - n_{5j^{\prime} - 4}) + \sum_{j=1}^{p} \sum_{j^{\prime}=1}^{j} (n_{5j-1} - n_{5j-2})(n_{5j^{\prime}-4} - n_{5j^{\prime} - 5}) \nonumber \\
& - \sum_{j=1}^{p} \sum_{j^{\prime} = 1}^{p} (n_{5j-3} - n_{5j-4})(n_{5j^{\prime}-5} - n_{5j^{\prime} - 6}) + \sum_{j=1}^{p} \sum_{j^{\prime} = 1}^{j} (n_{5j} - n_{5j-1})(n_{5j^{\prime}-4} - n_{5j^{\prime} - 5}).
\end{align}

\noindent Finally, using (\ref{step0}) and (\ref{step1}), then reindexing and simplifying gives the eighth term in $b_1(\underline{n})$, 

\begin{align}  \label{lastb1}
& \sum_{j=1}^{5p} \sigma_j \Bigl( \sum_{k=1}^{r^{\prime}(j)} (n_{i_k} - n_{i_k - 1}) \Bigr) n_{j-1} \nonumber \\
& = \sum_{j=1}^{p} \sigma_{5j-4} \Bigl( \sum_{k=1}^{4p+j} (n_{i_k} - n_{i_k - 1}) \Bigr) n_{5j-5} +  \sum_{j=1}^{p} \sigma_{5j-4} \Bigl( \sum_{k=1}^{2p+j} (n_{i_k} - n_{i_k - 1}) \Bigr) n_{5j-4} \nonumber \\
& +  \sum_{j=1}^{p-1} \sigma_{5j-2} \Bigl( \sum_{k=1}^{j} (n_{i_k} - n_{i_k - 1}) \Bigr) n_{5j-3} +  \sum_{j=1}^{p} \sigma_{5j} \Bigl( \sum_{k=1}^{4p-j+1} (n_{i_k} - n_{i_k - 1}) \Bigr) n_{5j-1} \nonumber \\
& + \sum_{j=1}^{p} \sigma_{5j-1} \Bigl( \sum_{k=1}^{2p-j+1} (n_{i_k} - n_{i_k - 1}) \Bigr) n_{5j-2} \nonumber \\
& = \sum_{j=1}^{p} \Biggl( \sum_{k=1}^{p} (n_{5k-2} - n_{5k-3}) + \sum_{k=p+1}^{2p} (n_{10p-5k+4} - n_{10p-5k+3}) + \sum_{k=2p+1}^{3p} (n_{5k-10p-3} - n_{5k-10p-4}) \nonumber \\
& + \sum_{k=3p+1}^{4p} (n_{20p-5k+5} - n_{20p-5k+4}) + \sum_{k=4p+1}^{4p+j} (n_{5k-20p-4} - n_{5k-20p-5}) \Biggr) n_{5j-5} \nonumber \\
& + \sum_{j=1}^{p} \Biggl( \sum_{k=1}^{p} (n_{5k-2} - n_{5k-3}) + \sum_{k=p+1}^{2p} (n_{10p-5k+4} - n_{10p-5k+3}) \nonumber \\
& + \sum_{k=2p+1}^{2p+j} (n_{5k-10p-3} - n_{5k-10p-4}) \Biggr) n_{5j-4} + \sum_{j=1}^{p} \Biggl( \sum_{k=1}^{j} (n_{5k-2} - n_{5k-3}) \Biggr) n_{5j-3} \nonumber \\
& - \sum_{j=1}^{p} \Biggl( \sum_{k=1}^{p} (n_{5k-2} - n_{5k-3}) + \sum_{k=p+1}^{2p} (n_{10p-5k+4} - n_{10p-5k+3}) + \sum_{k=2p+1}^{3p} (n_{5k-10p-3} - n_{5k-10p-4}) \nonumber \\
& + \sum_{k=3p+1}^{4p-j+1} (n_{20p-5k+5} - n_{20p-5k+4}) \Biggr) n_{5j-1} \nonumber \\
& - \sum_{j=1}^{p} \Biggl( \sum_{k=1}^{p} (n_{5k-2} - n_{5k-3}) + \sum_{k=p+1}^{2p-j+1} (n_{10p-5k+4} - n_{10p-5k+3} ) \Biggr) n_{5j-2} \nonumber \\
& = \sum_{j=1}^{p} \Biggl( \sum_{j^{\prime}=j+1}^{p} (n_{5j^{\prime}} - n_{5j^{\prime} - 4}) + n_{5j}) \Biggr) n_{5j-5} \nonumber \\
& + \sum_{j=1}^{p} \Biggl( \sum_{j^{\prime} = 1}^{j} (n_{5j^{\prime} - 1} - n_{5j^{\prime} - 4} ) + \sum_{j^{\prime}=j+1}^{p} (n_{5j^{\prime} - 1} - n_{5j^{\prime} - 3}) \Biggr) n_{5j-4} \nonumber \\
& + \sum_{j=1}^{p} \Biggl( \sum_{j'=1}^{j} (n_{5j^{\prime} - 2} - n_{5j^{\prime}-3}) \Biggr) n_{5j-3} \nonumber \\
& - \sum_{j=1}^{p} \Biggl( \sum_{j^{\prime} = 1}^{j-1} (n_{5j^{\prime} - 1} - n_{5j^{\prime} - 4}) + \sum_{j^{\prime}=j}^{p} (n_{5j^{\prime}} - n_{5j^{\prime} - 4}) \Biggr) n_{5j-1} \nonumber \\
& - \sum_{j=1}^{p} \Biggl( \sum_{j^{\prime}=1}^{j-1} (n_{5j^{\prime} - 2} - n_{5j^{\prime} - 3}) + \sum_{j^{\prime}=j}^{p} (n_{5j^{\prime} - 1} - n_{5j^{\prime} - 3}) \Biggr) n_{5j-2}.
\end{align}

\noindent Thus, combining (\ref{a})--(\ref{lastb1}) implies that $b_1({\underline{n}}) + b_2({\underline{n}})$ equals

\begin{align*}
& p + \sum_{j=1}^{p-1} n_{5j}  + \sum_{j=1}^{p} (n_{5j-4} + n_{5j-3} - n_{5j-2} - n_{5j-1}) \\
& - \sum_{j=1}^{p} \sum_{j^{\prime}=1}^{j} (n_{5j-2} - n_{5j-3})(n_{5j^{\prime} - 6} - n_{5j^{\prime} - 7}) - \sum_{j=1}^{p} \sum_{j^{\prime}=1}^{j} (n_{5j-2} - n_{5j-3})(n_{5j^{\prime} -5} - n_{5j^{\prime} - 6}) \\
& + \sum_{j=1}^{p} \sum_{j^{\prime} = 1}^{j} (n_{5j-1} - n_{5j-2})(n_{5j^{\prime} - 3} - n_{5j^{\prime} - 4}) + \sum_{j=1}^{p} \sum_{j^{\prime}=1}^{j} (n_{5j-1} - n_{5j-2})(n_{5j^{\prime}-4} - n_{5j^{\prime} - 5}) \\
& - \sum_{j=1}^{p} \sum_{j^{\prime} = 1}^{p} (n_{5j-3} - n_{5j-4})(n_{5j^{\prime}-5} - n_{5j^{\prime} - 6}) + \sum_{j=1}^{p} \sum_{j^{\prime} = 1}^{j} (n_{5j} - n_{5j-1})(n_{5j^{\prime}-4} - n_{5j^{\prime} - 5}) \\
& + \sum_{j=1}^{p} \Biggl( \sum_{j^{\prime}=j+1}^{p} (n_{5j^{\prime}} - n_{5j^{\prime} - 4}) + n_{5j}) \Biggr) n_{5j-5} \\
& + \sum_{j=1}^{p} \Biggl( \sum_{j^{\prime} = 1}^{j} (n_{5j^{\prime} - 1} - n_{5j^{\prime} - 4} ) + \sum_{j^{\prime}=j+1}^{p} (n_{5j^{\prime} - 1} - n_{5j^{\prime} - 3}) \Biggr) n_{5j-4}  \\
& + \sum_{j=1}^{p} \Biggl( \sum_{j'=1}^{j} (n_{5j^{\prime} - 2} - n_{5j^{\prime}-3}) \Biggr) n_{5j-3} \\
& - \sum_{j=1}^{p} \Biggl( \sum_{j^{\prime} = 1}^{j-1} (n_{5j^{\prime} - 1} - n_{5j^{\prime} - 4}) + \sum_{j^{\prime}=j}^{p} (n_{5j^{\prime}} - n_{5j^{\prime} - 4}) \Biggr) n_{5j-1} \\
& - \sum_{j=1}^{p} \Biggl( \sum_{j^{\prime}=1}^{j-1} (n_{5j^{\prime} - 2} - n_{5j^{\prime} - 3}) + \sum_{j^{\prime}=j}^{p} (n_{5j^{\prime} - 1} - n_{5j^{\prime} - 3}) \Biggr) n_{5j-2}.
\end{align*}

\section{Proof of Theorem \ref{main1}}

\begin{proof}[Proof of Theorem \ref{main1}] Using Lemmas \ref{l10} and \ref{l11}, one can check that for $l=4mp + 2p - 1$ and $t=4mp - 1$

\begin{equation} \label{a1}
a({\underline{n}}) = \sum_{j=1}^{p - 1} n_{(2m+1)j} + \sum_{j=1}^{p} n_{(2m+1)j - (m+1)} + p - 1
\end{equation}

\noindent and $b_1({\underline{n}}) + b_2({\underline{n}})$ equals

\begin{align}
& 1 - p + n_{(2m+1)p - 1} + \sum_{j=1}^{p} \Biggl( \sum_{i=1}^{m-1} n_{(2m+1)j - 2m + i - 1} - \sum_{i=m+1}^{2m} n_{(2m+1)j - 2m + i - 1} \Biggr) \nonumber \\
& + \sum_{j=1}^{p} \sum_{j'=1}^{j} \sum_{k=1}^{m} \sum_{k'=1}^{k} (n_{(2m+1)j - k} - n_{(2m+1)j - k - 1})(n_{(2m+1)j' + k - 2m - k'} - n_{(2m+1)j' + k - 2m - k' - 1}) \nonumber \\
& - \sum_{j=1}^{p} \sum_{j'=1}^{j} \sum_{k=1}^{m} \sum_{k'=1}^{m-k+1} (n_{(2m+1)j - m - k} - n_{(2m+1)j - m - k - 1})(n_{(2m+1)j' - 2m - k'} - n_{(2m+1)j' - 2m - k' - 1}) \nonumber \\
& + \sum_{s=1}^{m} \sum_{j=1}^{p} \Biggl( \sum_{j'=1}^{j-1} (n_{(2m+1)j' - s} - n_{(2m+1)j' - 2m + s -1}) \nonumber \\
& \quad \quad + \sum_{j'=j}^{p} (n_{(2m+1)j' - s} - n_{(2m+1)j' - 2m + s - 2}) \Biggr) n_{(2m+1)j - 2m + s -2} \nonumber \\
& - \sum_{j=1}^{p - 1} \Biggl( \sum_{j'=j+1}^{p} (n_{(2m+1)j' - 1} - n_{(2m+1)j' - (2m+1)}) + n_{(2m+1)j} \Biggr) n_{(2m+1)j - 1} \nonumber \\
& - \sum_{s=1}^{m} \sum_{j=1}^{p} \Biggl( \sum_{j'=1}^{j} (n_{(2m+1)j' - m + s-1} - n_{(2m+1)j' - m - s}) \\ \nonumber 
& \quad \quad + \sum_{j' = j+1}^{p} (n_{(2m+1)j' - m + s -2} - n_{(2m+1)j' - m - s}) \Biggr) n_{(2m+1)j - m + s -2}.
\end{align}

\noindent Also, by (\ref{kappa}) and (\ref{tau}), $X({\underline{n}})$ equals

\begin{align}  \label{twist3}
& (-1)^{n_{(2m+1)p - 1}} q^{-Nn_{(2m+1)p - 1}} \frac{(q)_{N-1} (q)_{n_{(2m+1)p - 1}}}{(q)_{N - n_{(2m+1)p - 1} - 1}} \prod_{j=1}^{p} \frac{(-1)^{n_{(2m+1)j - 1} - n_{(2m+1)j-m-1}} q^{\frac{1}{2} \sum\limits_{s=1}^{m} S(m,j,s)}}{\displaystyle \prod\limits_{s=1}^{2m} (q)_{n_{(2m+1)j-2m+s-1} - n_{(2m+1)j-2m+s-2}}} \nonumber \\
& \times \prod_{j=1}^{p - 1} \frac{(-1)^{n_{(2m+1)j} - n_{(2m+1)j-1}}}{(q)_{n_{(2m+1)j} - n_{(2m+1)j-1}}}
\end{align}
where

\begin{equation} \label{smjs}
S(m,j,s) := (n_{(2m+1)j-2m+s-1} - n_{(2m+1)j-2m+s-2})(n_{(2m+1)j-2m+s-1} - n_{(2m+1)j-2m+s-2} + 1).
\end{equation}

We first consider the case $m=1$.   Upon comparing (\ref{spt}) and (\ref{a1})--(\ref{smjs}) with (\ref{JNK-m-p}) and then simplifying, it suffices to prove that 

\begin{align} \label{lhsm1}
& \sum_{j=1}^{p} \sum_{j'=1}^{j} (n_{3j-1} - n_{3j-2})(n_{3j'-2} - n_{3j'-3}) - \sum_{j=1}^{p} \sum_{j'=1}^{j} (n_{3j-2} - n_{3j-3})(n_{3j'-3} - n_{3j'-4}) \nonumber \\
& + \sum_{j=1}^{p} \Biggl( \sum_{j'=1}^{j-1} (n_{3j'-1} - n_{3j'-2}) + \sum_{j'=j}^{p} (n_{3j'-1} - n_{3j'-3})\Biggr) n_{3j-3} \nonumber \\
& - \sum_{j=1}^{p-1} \Biggl( \sum_{j'=j+1}^{p} (n_{3j'-1} - n_{3j'-3}) + n_{3j} \Biggr) n_{3j-1} - \sum_{j=1}^{p} \sum_{j'=1}^{j} (n_{3j'-1} - n_{3j'-2}) n_{3j-2} \nonumber \\
& - \sum_{j=1}^{p} n_{3j-2} n_{3j-3} 
\end{align}

\noindent equals

\begin{equation} \label{newrhsm=1}
\displaystyle{\sum_{\substack{1 \leq i < j \leq 3p-1 \\ 3 \nmid i \\ j \not \equiv 1 \pmod*{3}}} \epsilon_{i,j,1} n_i n_j} - \displaystyle \sum_{i=1}^{3p-2} n_i n_{i+1}
\end{equation} 

\noindent where $\epsilon_{i,j,m}$ is given by (\ref{epsilondef}). Here, we have used the fact that

\begin{equation} \label{s1m=1}
- \sum_{j=1}^{p-1} n_{3j-1} =  \sum_{i=1}^{3p-2} \gamma_{i,1} n_i + \sum_{i=1}^{p-1} n_{3i},
\end{equation}

\noindent where $\gamma_{i,m}$ is given by (\ref{gammadef}), together with the identities

\begin{equation}  \label{s2m=1}
\frac{1}{2} \sum_{j=1}^{p} S(1,j,1) = \sum_{j=1}^{p} \binom{n_{3j-3}}{2} + \sum_{j=1}^{p} \binom{n_{3j-2} + 1}{2} - n_{3j-2} n_{3j-3} 
\end{equation}

\noindent where $S(m,j,s)$ is given by (\ref{smjs}) and

\begin{equation} \label{s3m=1}
\sum_{j=1}^{p} \binom{n_{3j - 3}}{2} + \sum_{i=1}^{p-1} n_{3i}= \sum_{i=1}^{p-1} \binom{n_{3i} + 1}{2}.
\end{equation}

We now explain how to proceed from (\ref{lhsm1}) to (\ref{newrhsm=1}). After taking out the $j'=j$ term from the fourth sum in the third line of (\ref{lhsm1}) and simplifying, we obtain

\begin{align} \label{lhsm12}
& \sum_{j=1}^{p} \sum_{j'=1}^{j} n_{3j-1} n_{3j'-2} + \sum_{j=1}^{p} \sum_{j'=1}^{j-1} n_{3j'-1} n_{3j-3} - \sum_{j=1}^{p} \sum_{j'=1}^{j-1} n_{3j'-2} n_{3j-3} - \sum_{j=1}^{p-1} \sum_{j'=j+1}^{p} n_{3j'-1} n_{3j-1} \nonumber \\
& - \sum_{j=1}^{p-1} n_{3j} n_{3j-1} - \sum_{j=1}^{p} n_{3j-2} n_{3j-3} - \sum_{j=1}^{p} n_{3j-1} n_{3j-2}.
\end{align}

\noindent The first line of (\ref{lhsm12}) corresponds to the first sum in (\ref{newrhsm=1}); namely, the first two sums correspond to $(i,j) \equiv (i, -i) \pmod{3}$ and $(i,j) \equiv (i,-i-1) \pmod{3}$, respectively, while the second two sums correspond to $(i,j) \equiv (i,i-1) \pmod{3}$ and $(i,j) \equiv (i,i) \pmod{3}$, respectively. The three sums in the second line of (\ref{lhsm12}) match the second sum of (\ref{newrhsm=1}). Thus, we have proven that (\ref{lhsm1}) equals (\ref{newrhsm=1}).

We now turn to the general case $m \geq 2$.   Upon comparing (\ref{spt}) and (\ref{a1})--(\ref{smjs}) with (\ref{JNK-m-p}) and then simplifying, it suffices to prove that 

\begin{align} \label{lhs}
& \sum_{j=1}^{p} \sum_{j'=1}^{j} \sum_{k=1}^{m} \sum_{k'=1}^{k} (n_{(2m+1)j - k} - n_{(2m+1)j - k - 1})(n_{(2m+1)j' + k - 2m - k'} - n_{(2m+1)j' + k - 2m - k' - 1}) \nonumber \\
& - \sum_{j=1}^{p} \sum_{j'=1}^{j} \sum_{k=1}^{m} \sum_{k'=1}^{m-k+1} (n_{(2m+1)j - m - k} - n_{(2m+1)j - m - k - 1})(n_{(2m+1)j' - 2m - k'} - n_{(2m+1)j' - 2m - k' - 1}) \nonumber \\
& + \sum_{s=1}^{m} \sum_{j=1}^{p} \Biggl( \sum_{j'=1}^{j-1} (n_{(2m+1)j' - s} - n_{(2m+1)j' - 2m + s -1})  \nonumber \\
& \quad \quad + \sum_{j'=j}^{p} (n_{(2m+1)j' - s} - n_{(2m+1)j' - 2m + s - 2}) \Biggr) n_{(2m+1)j - 2m + s -2} \nonumber \\
& - \sum_{j=1}^{p - 1} \Biggl( \sum_{j'=j+1}^{p} (n_{(2m+1)j' - 1} - n_{(2m+1)j' - (2m+1)}) + n_{(2m+1)j} \Biggr) n_{(2m+1)j - 1} \nonumber \\
& - \sum_{s=1}^{m} \sum_{j=1}^{p} \Biggl( \sum_{j'=1}^{j} (n_{(2m+1)j' - m + s-1} - n_{(2m+1)j' - m - s}) \\ \nonumber 
& \quad \quad + \sum_{j' = j+1}^{p} (n_{(2m+1)j' - m + s -2} - n_{(2m+1)j' - m - s}) \Biggr) n_{(2m+1)j - m + s -2} 
\\ \nonumber
& + \sum_{j=1}^{p} \Biggl[ \binom{n_{(2m+1)j-2m} + 1}{2} - n_{(2m+1)j-2m} n_{(2m+1)j-2m-1} + \binom{n_{(2m+1)j-m-2}}{2} \nonumber \\
& \quad \quad - n_{(2m+1)j-m-1} n_{(2m+1)j-m-2} + \frac{1}{2} \sum_{s=2}^{m-1} S(m,j,s) \Biggr] \nonumber
\end{align}

\noindent equals

\begin{equation} \label{newrhs}
\displaystyle{\sum_{\substack{1 \leq i < j \leq (2m+1)p-1 \\ (2m+1) \nmid i \\ j \not \equiv m \pmod*{2m+1}}} \epsilon_{i,j,m} n_i n_j} - \displaystyle \sum_{i=1}^{(2m+1)p-2} n_i n_{i+1}.
\end{equation} 

\noindent Here, we have used the fact that

\begin{equation} \label{s1}
\begin{aligned}
\sum_{j=1}^{p} \Biggl( \sum_{i=1}^{m-1} n_{(2m+1)j - 2m + i-1} & - \sum_{i=m+1}^{2m-1} n_{(2m+1)j - 2m+i-1} \Biggr)  - \sum_{j=1}^{p-1} n_{(2m+1)j-1} \\
& =  \sum_{i=1}^{(2m+1)p-2} \gamma_{i,m} n_i + \sum_{i=1}^{p-1} n_{(2m+1)i},
\end{aligned}
\end{equation}

\noindent together with the identities

\begin{align}  \label{s2}
\frac{1}{2} \sum_{j=1}^{p} \sum_{s=1}^{m} S(m,j,s) & = \sum_{j=1}^{p} \binom{n_{(2m+1)j - 2m-1}}{2} + \binom{n_{(2m+1)j-m-1} + 1}{2} \nonumber \\
& + \sum_{j=1}^{p} \Biggl[ \binom{n_{(2m+1)j-2m} + 1}{2} - n_{(2m+1)j-2m} n_{(2m+1)j-2m-1} + \binom{n_{(2m+1)j-m-2}}{2} \nonumber \\
& \quad \quad - n_{(2m+1)j-m-1} n_{(2m+1)j-m-2} + \frac{1}{2} \sum_{s=2}^{m-1} S(m,j,s) \Biggr] 
\end{align}

\noindent and

\begin{equation} \label{s3}
\sum_{j=1}^{p} \binom{n_{(2m+1)j - 2m-1}}{2} + \sum_{i=1}^{p-1} n_{(2m+1)i}= \sum_{i=1}^{p-1} \binom{n_{(2m+1)i} + 1}{2}.
\end{equation}

We now sketch how to proceed from (\ref{lhs}) to (\ref{newrhs}). For $1 \leq i \leq 9$, let $L_i$ denote the $i$th line of \eqref{lhs}. First note that 
\begin{equation} \label{l8l9}
L_8 + L_9 = \sum_{j=0}^{p-1} \sum_{i=1}^{m-1} n_{(2m+1)j + i}^2 - \sum_{j=0}^{p-1} \sum_{i=1}^{m} n_{(2m+1)j + i}n_{(2m+1)j + i-1}.
\end{equation}
Next, the sum over $k'$ in both $L_1$ and $L_2$ telescopes, and we obtain 
\begin{equation} \label{L1expanded}
\begin{aligned}
L_1 &= \sum_{k=1}^m \sum_{j=1}^p \sum_{j'=1}^j (n_{(2m+1)j - k} - n_{(2m+1)j-k-1})n_{(2m+1)j' + k -2m-1} \\
& \qquad \qquad - \sum_{k=1}^m \sum_{j=1}^p \sum_{j'=1}^j (n_{(2m+1)j - k} - n_{(2m+1)j-k-1})n_{(2m+1)j'-2m-1}
\end{aligned}
\end{equation}
and
\begin{equation} \label{L2expanded}
\begin{aligned}
L_2 &= -\sum_{k=1}^m \sum_{j=1}^p \sum_{j'=1}^j (n_{(2m+1)j - m - k} - n_{(2m+1)j-m-k-1})n_{(2m+1)j' -2m-1} \\
& \qquad \qquad + \sum_{k=1}^m \sum_{j=1}^p \sum_{j'=1}^j (n_{(2m+1)j - m- k} - n_{(2m+1)j-m-k-1})n_{(2m+1)j'-3m + k-2}.
\end{aligned}
\end{equation}
Now the sum over $k$ in the second line of \eqref{L1expanded} and the first line of \eqref{L2expanded} both telescope and so
\begin{equation} \label{L1expandedmore}
\begin{aligned}
L_1 &= \sum_{k=1}^m \sum_{j=1}^p \sum_{j'=1}^j n_{(2m+1)j - k}n_{(2m+1)j' + k -2m-1} - \sum_{k=1}^m \sum_{j=1}^p \sum_{j'=1}^jn_{(2m+1)j-k-1}n_{(2m+1)j' + k -2m-1} \\
&\qquad + \sum_{j=1}^p \sum_{j'=1}^jn_{(2m+1)j - m-1}n_{(2m+1)j'-2m-1} - \sum_{j=1}^p \sum_{j'=1}^j n_{(2m+1)j-1}n_{(2m+1)j'-2m-1}
\end{aligned}
\end{equation}
and
\begin{equation} \label{L2expandedmore}
\begin{aligned}
L_2 &= \sum_{k=1}^m \sum_{j=1}^p \sum_{j'=1}^j n_{(2m+1)j - m- k}n_{(2m+1)j'-3m+k-2} \\
& \qquad - \sum_{k=1}^m \sum_{j=1}^p \sum_{j'=1}^j n_{(2m+1)j-m-k-1}n_{(2m+1)j'-3m+k-2} + \sum_{j=1}^p \sum_{j'=1}^j n_{(2m+1)j - 2m-1}n_{(2m+1)j' -2m-1} \\
& \qquad \qquad - \sum_{j=1}^p \sum_{j'=1}^j n_{(2m+1)j-m-1}n_{(2m+1)j' -2m-1}.
\end{aligned}
\end{equation}
Observe that the third sum in \eqref{L1expandedmore} and the fourth sum in \eqref{L2expandedmore} cancel.    Moreover, if we take $s=1$ in the triple sum in $L_4$,
\begin{equation}
\sum_{s=1}^m \sum_{j=1}^p \sum_{j'=j}^p (n_{(2m+1)j' -s} - n_{(2m+1)j'-2m+s-2})n_{(2m+1)j - 2m + s-2},
\end{equation}
and exchange $j$ and $j'$ we see that this cancels with the fourth sum in \eqref{L1expandedmore} and the third sum in \eqref{L2expandedmore}. Putting this and (\ref{l8l9}) together and expanding all of the sums we find that \eqref{lhs} equals

\begin{align} \label{lhs1}
&\sum_{k=1}^m \sum_{j=1}^p \sum_{j'=1}^j n_{(2m+1)j - k}n_{(2m+1)j' + k -2m-1} - \sum_{k=1}^m \sum_{j=1}^p \sum_{j'=1}^jn_{(2m+1)j-k-1}n_{(2m+1)j' + k -2m-1} \nonumber \\
&+ \sum_{k=1}^m \sum_{j=1}^p \sum_{j'=1}^j n_{(2m+1)j - m- k}n_{(2m+1)j'-3m+k-2} - \sum_{k=1}^m \sum_{j=1}^p \sum_{j'=1}^j n_{(2m+1)j-m-k-1}n_{(2m+1)j'-3m+k-2} \nonumber \\
& + \sum_{s=1}^{m} \sum_{j=1}^{p} \sum_{j'=1}^{j-1} n_{(2m+1)j' - s} n_{(2m+1)j - 2m + s-2} - \sum_{s=1}^{m} \sum_{j=1}^{p} \sum_{j'=1}^{j-1} n_{(2m+1)j' - 2m + s-1} n_{(2m+1)j - 2m + s-2} \nonumber \\
& + \sum_{s=2}^{m} \sum_{j=1}^{p} \sum_{j'=j}^{p} n_{(2m+1)j'-s} n_{(2m+1)j - 2m + s-2} - \sum_{s=2}^{m} \sum_{j=1}^{p} \sum_{j'=j}^{p} n_{(2m+1)j' - 2m + s-2} n_{(2m+1)j - 2m + s-2} \nonumber \\
& - \sum_{j=1}^{p-1} \sum_{j'=j+1}^{p} n_{(2m+1)j'-1} n_{(2m+1)j-1} + \sum_{j=1}^{p-1} \sum_{j'=j+1}^{p} n_{(2m+1)j' - (2m+1)} n_{(2m+1)j-1}\nonumber \\
& - \sum_{j=1}^{p-1} n_{(2m+1)j} n_{(2m+1)j - 1} \nonumber \\
& - \sum_{s=1}^{m} \sum_{j=1}^{p} \sum_{j'=1}^{j} n_{(2m+1)j' - m + s-1} n_{(2m+1)j - m + s-2} + \sum_{s=1}^{m} \sum_{j=1}^{p} \sum_{j'=1}^{j} n_{(2m+1)j' - m - s} n_{(2m+1)j - m + s-2} \nonumber \\
&  - \sum_{s=1}^{m} \sum_{j=1}^{p} \sum_{j'=j+1}^{p} n_{(2m+1)j' - m + s-2} n_{(2m+1)j - m + s-2} + \sum_{s=1}^{m} \sum_{j=1}^{p} \sum_{j'=j+1}^{p} n_{(2m+1)j' - m - s} n_{(2m+1)j - m + s-2} \nonumber \\
& + \sum_{j=0}^{p-1} \sum_{i=1}^{m-1} n_{(2m+1)j + i}^2 - \sum_{j=0}^{p-1} \sum_{i=1}^{m} n_{(2m+1)j + i} n_{(2m+1)j + i-1}.
\end{align}

\noindent In the second sum on the fourth line of (\ref{lhs1}), we exchange $j$ and $j'$ and reindex to obtain

\begin{equation*}
- \sum_{s=2}^{m} \sum_{j=1}^{p} \sum_{j'=1}^{j} n_{(2m+1)j - 2m + s-2} n_{(2m+1)j' - 2m + s-2}.
\end{equation*}

\noindent We then take out the term $j'=j$ and shift the indices in this term by $j \to j+1$ and $s \to s+1$ to cancel with the first sum on the last line of (\ref{lhs1}). In the second line of (\ref{lhs1}), perform the shift $j' \to j' + 1$ and start the sum at $j'=1$ (as $j'=0$ gives $0$) to obtain

\begin{equation} \label{secondline}
\sum_{k=1}^{m} \sum_{j=1}^{p} \sum_{j'=1}^{j-1} n_{(2m+1)j - m - k} n_{(2m+1)j' - m + k-1} - \sum_{k=1}^{m} \sum_{j=1}^{p} \sum_{j'=1}^{j-1} n_{(2m+1)j - m - k -1} n_{(2m+1)j' - m + k-1}.
\end{equation}

\noindent Now, in the second sum of the penultimate line of (\ref{lhs1}), we exchange $j$ and $j'$ and reindex, shift by $s \to s+1$, then remove the $s=0$ term. Note that what remains cancels with the second sum in (\ref{secondline}) after removing the $k=m$ term. In total, this yields that (\ref{lhs}) equals

\begin{align} \label{lhs2}
&\sum_{k=1}^m \sum_{j=1}^p \sum_{j'=1}^j n_{(2m+1)j - k}n_{(2m+1)j' + k -2m-1} - \sum_{k=1}^m \sum_{j=1}^p \sum_{j'=1}^jn_{(2m+1)j-k-1}n_{(2m+1)j' + k -2m-1} \nonumber  \\
&+ \sum_{k=1}^m \sum_{j=1}^p \sum_{j'=1}^{j-1} n_{(2m+1)j - m- k} n_{(2m+1)j'-m+k-1} - \sum_{j=1}^{p} \sum_{j'=1}^{j-1} n_{(2m+1)j - 2m -1} n_{(2m+1)j'-1} \nonumber \\
& + \sum_{s=1}^{m} \sum_{j=1}^{p} \sum_{j'=1}^{j-1} n_{(2m+1)j' - s} n_{(2m+1)j - 2m + s-2} - \sum_{s=1}^{m} \sum_{j=1}^{p} \sum_{j'=1}^{j-1} n_{(2m+1)j' - 2m + s-1} n_{(2m+1)j - 2m + s-2} \nonumber \\
& + \sum_{s=2}^{m} \sum_{j=1}^{p} \sum_{j'=j}^{p} n_{(2m+1)j'-s} n_{(2m+1)j - 2m + s-2} - \sum_{s=2}^{m} \sum_{j=1}^{p} \sum_{j'=1}^{j-1} n_{(2m+1)j' - 2m + s-2} n_{(2m+1)j - 2m + s-2} \nonumber \\
& - \sum_{j=1}^{p-1} \sum_{j'=j+1}^{p} n_{(2m+1)j'-1} n_{(2m+1)j-1} + \sum_{j=1}^{p-1} \sum_{j'=j+1}^{p} n_{(2m+1)j' - (2m+1)} n_{(2m+1)j-1} \nonumber \\
& - \sum_{j=1}^{p-1} n_{(2m+1)j} n_{(2m+1)j - 1} \nonumber \\
& - \sum_{s=1}^{m} \sum_{j=1}^{p} \sum_{j'=1}^{j} n_{(2m+1)j' - m + s-1} n_{(2m+1)j - m + s-2} + \sum_{s=1}^{m} \sum_{j=1}^{p} \sum_{j'=1}^{j} n_{(2m+1)j' - m - s} n_{(2m+1)j - m + s-2} \nonumber \\
&  - \sum_{s=1}^{m} \sum_{j=1}^{p} \sum_{j'=j+1}^{p} n_{(2m+1)j' - m + s-2} n_{(2m+1)j - m + s-2} + \sum_{j=1}^{p} \sum_{j'=1}^{j-1} n_{(2m+1)j - m - 1} n_{(2m+1)j' - m -1}\nonumber \\
& - \sum_{j=0}^{p-1} \sum_{i=1}^{m} n_{(2m+1)j + i} n_{(2m+1)j + i-1}.
\end{align}

\noindent We now simplify further. The $s=1$ term of the first sum in the penultimate line cancels with the second sum in the same line. Remove the $j'=j$ term from the first sum in the seventh line and write it in the last line. The $s=1$ term of the remaining triple sum cancels with the $k=1$ term of the first sum on the second line. The first sum on the fourth line cancels with the second sum of the first line once we remove the $k=m$ term. This $k=m$ term then cancels with the $s=1$ term of the second sum of the seventh line. The first sum in the fifth line is the $s=m+1$ term of the first sum in the penultimate line. The second sum in the fifth line cancels with the second sum in the second line. Finally, the sum in the sixth line is the $i=0$ term in the last line. Thus, (\ref{lhs}) equals

\begin{align} \label{lhs3}
& \sum_{k=1}^{m} \sum_{j=1}^{p} \sum_{j'=1}^{j} n_{(2m+1)j - k} n_{(2m+1)j' + k - 2m - 1} + \sum_{k=2}^{m} \sum_{j=1}^{p} \sum_{j'=1}^{j-1} n_{(2m+1)j - m - k} n_{(2m+1)j' - m + k-1} \nonumber \\
&+ \sum_{s=1}^{m} \sum_{j=1}^{p} \sum_{j'=1}^{j-1} n_{(2m+1)j'-s} n_{(2m+1)j - 2m + s-2} + \sum_{s=2}^{m} \sum_{j=1}^{p} \sum_{j'=1}^{j} n_{(2m+1)j' - m - s} n_{(2m+1)j - m + s-2} \nonumber \\ 
& - \sum_{s=1}^{m} \sum_{j=1}^{p} \sum_{j'=1}^{j-1} n_{(2m+1)j' - 2m + s-1} n_{(2m+1)j - 2m + s-2} \nonumber \\
& - \sum_{s=2}^{m} \sum_{j=1}^{p} \sum_{j'=1}^{j-1} n_{(2m+1)j' - 2m + s-2} n_{(2m+1)j - 2m + s-2} \nonumber \\
& - \sum_{s=2}^{m} \sum_{j=1}^{p} \sum_{j'=1}^{j-1} n_{(2m+1)j' - m + s-1} n_{(2m+1)j - m + s-2} - \sum_{s=1}^{m} \sum_{j=1}^{p} \sum_{j'=j+1}^{p} n_{(2m+1)j' - m + s-1} n_{(2m+1)j - m + s-1} \nonumber \\
 & - \sum_{s=1}^{m} \sum_{j=1}^{p} n_{(2m+1)j - m + s-1} n_{(2m+1)j - m + s-2} - \sum_{j=0}^{p-1} \sum_{i=0}^{m} n_{(2m+1)j + i} n_{(2m+1)j + i-1}.
\end{align}

\noindent Now we see that this is equal to (\ref{newrhs}) as follows. The first five lines of (\ref{lhs3}) correspond to the first term in (\ref{newrhs}); namely, the first line of (\ref{lhs3}) corresponds to $(i,j) \equiv (i,-i) \pmod{2m+1}$ while the second line corresponds to $(i,j) \equiv (i, -i-1) \pmod{2m+1}$. The first sums in the third and fifth lines correspond to $(i,j) \equiv (i,i-1) \pmod{2m+1}$ while the sum
in the fourth line and the second sum in the fifth line correspond to $(i,j) \equiv (i,i) \pmod{2m+1}$. Finally, the sixth line of (\ref{lhs3}) matches the second sum of (\ref{newrhs}). Thus, we have proven that (\ref{lhs}) equals (\ref{newrhs}).

\end{proof}

\section{Proof of Theorem \ref{main2}}

\begin{proof}[Proof of Theorem \ref{main2}] As (\ref{JNK-mp}) reduces to (\ref{t22p+1}) when $m=0$ and this case was proven in \cite{hikami1}, we assume that $m \geq 1$. Using Lemmas \ref{l12} and \ref{l13}, one can check that for $l=4mp + 2p + 1$ and $t=4mp + 1$

\begin{equation} \label{a2}
a({\underline{n}}) = -\sum_{j=1}^{p - 1} n_{(2m+1)j} - \sum_{j=1}^{p} n_{(2m+1)j - (m)} - p
\end{equation}

\noindent and $b_1({\underline{n}}) + b_2({\underline{n}})$ equals

\begin{align}
& p + \sum_{j=1}^{p - 1} n_{(2m+1)j} + \sum_{j=1}^{p} \Biggl( \sum_{i=1}^{m} n_{(2m+1)j - 2m + i - 1} - \sum_{i=m+1}^{2m} n_{(2m+1)j - 2m + i - 1} \Biggr) \nonumber \\
& + \sum_{j=1}^{p} \sum_{j'=1}^{j} \sum_{k=1}^{m} \sum_{k'=1}^{k} (n_{(2m+1)j - k + 1} - n_{(2m+1)j - k}) (n_{(2m+1)j' + k - 2m - k'} - n_{(2m+1)j' + k - 2m - k' - 1}) \nonumber \\
& - \sum_{j=1}^{p} \sum_{j'=1}^{j} \sum_{k=1}^{m} \sum_{k'=1}^{m-k+1} 
(n_{(2m+1)j - m - k + 1} - n_{(2m+1)j - m - k})(n_{(2m+1)j' - 2m - k'} - n_{(2m+1)j' - 2m - k' - 1}) \nonumber \\
& + \sum_{s=1}^{m} \sum_{j=1}^{p} \Biggl( \sum_{j'=1}^{j} ( n_{(2m+1)j' - s} - n_{(2m+1)j' - 2m + s - 1} ) \nonumber \\
& \quad \quad + \sum_{j'=j+1}^{p} ( n_{(2m+1)j' - s} - n_{(2m+1)j' - 2m + s} ) \Biggr) n_{(2m+1)j - 2m + s -1} \nonumber \\
& + \sum_{j=1}^{p} \Biggl( \sum_{j'=j+1}^{p} (n_{(2m+1)j'} - n_{(2m+1)j' - 2m}) + n_{(2m+1)j} \Biggr) n_{(2m+1)j-(2m+1)} \nonumber \\
& - \sum_{s=1}^{m} \sum_{j=1}^{p} \Biggl( \sum_{j'=1}^{j-1} ( n_{(2m+1)j' - s} - n_{(2m+1)j'- 2m + s - 1} ) \nonumber \\
& \quad \quad + \sum_{j'=j}^{p} ( n_{(2m+1)j' - s +1} - n_{(2m+1)j' - 2m + s - 1} ) \Biggr) n_{(2m+1)j - s}.
\end{align}

\noindent Also, by (\ref{kappa}) and (\ref{tau}), $X({\underline{n}})$ equals

\begin{align} \label{twist4}
& (-1)^{n_{(2m+1)p}} q^{-Nn_{(2m+1)p}} \frac{(q)_{N-1} (q)_{n_{(2m+1)p}}}{(q)_{N - n_{(2m+1)p} - 1}} \prod_{j=1}^{p} \frac{(-1)^{n_{(2m+1)j} - n_{(2m+1)j - m}} q^{\frac{1}{2} \sum\limits_{s=1}^{m+1} S(m,j,s)}}{\displaystyle \prod\limits_{s=1}^{2m+1} (q)_{n_{(2m+1)j-2m+s-1} - n_{(2m+1)j-2m+s-2}}}
\end{align}
where $S(m,j,s)$ is given by (\ref{smjs}). 

Upon comparing (\ref{spt}) and (\ref{a2})--(\ref{twist4}) with (\ref{JNK-mp}) and then simplifying, it suffices to prove
that for $m \geq 1$

\begin{align} \label{newlhs2}
& \sum_{j=1}^{p} \sum_{j'=1}^{j} \sum_{k=1}^{m} \sum_{k'=1}^{k} (n_{(2m+1)j - k + 1} - n_{(2m+1)j - k}) (n_{(2m+1)j' + k - 2m - k'} - n_{(2m+1)j' + k - 2m - k' - 1}) \nonumber \\
& - \sum_{j=1}^{p} \sum_{j'=1}^{j} \sum_{k=1}^{m} \sum_{k'=1}^{m-k+1} 
(n_{(2m+1)j - m - k + 1} - n_{(2m+1)j - m - k})(n_{(2m+1)j' - 2m - k'} - n_{(2m+1)j' - 2m - k' - 1}) \nonumber \\
& + \sum_{s=1}^{m} \sum_{j=1}^{p} \Biggl( \sum_{j'=1}^{j} ( n_{(2m+1)j' - s} - n_{(2m+1)j' - 2m + s - 1} ) \nonumber \\
& \quad \quad + \sum_{j'=j+1}^{p} ( n_{(2m+1)j' - s} - n_{(2m+1)j' - 2m + s} ) \Biggr) n_{(2m+1)j - 2m + s -1} \nonumber \\
& + \sum_{j=1}^{p} \Biggl( \sum_{j'=j+1}^{p} (n_{(2m+1)j'} - n_{(2m+1)j' - 2m}) + n_{(2m+1)j} \Biggr) n_{(2m+1)j-(2m+1)} \nonumber \\
& - \sum_{s=1}^{m} \sum_{j=1}^{p} \Biggl( \sum_{j'=1}^{j-1} ( n_{(2m+1)j' - s} - n_{(2m+1)j'- 2m + s - 1} ) \nonumber \\
& \quad \quad + \sum_{j'=j}^{p} ( n_{(2m+1)j' - s +1} - n_{(2m+1)j' - 2m + s - 1} ) \Biggr) n_{(2m+1)j - s} \nonumber \\
& + \sum_{j=1}^{p} \Biggl[ \binom{n_{(2m+1)j-2m} + 1}{2} - n_{(2m+1)j-2m} n_{(2m+1)j-2m-1} + \binom{n_{(2m+1)j-m-1}}{2} \nonumber \\
& \quad \quad - n_{(2m+1)j-m} n_{(2m+1)j-m-1} + \frac{1}{2} \sum_{s=2}^{m} S(m,j,s) \Biggr]
\end{align}

\noindent equals

\begin{equation} \label{Deltasum}
\displaystyle{\sum_{\substack{1 \leq i < j \leq (2m+1)p \\ (2m+1) \nmid i \\ j \not \equiv m+1 \pmod*{2m+1}}} \Delta_{i,j,m} n_i n_j}
\end{equation}

\noindent where $\Delta_{i,j,m}$ is given by (\ref{Deltadef}). Here, we have used (\ref{s3}),

\begin{align}  \label{news2}
\frac{1}{2} \sum_{j=1}^{p} \sum_{s=1}^{m+1} S(m,j,s) & = \sum_{j=1}^{p} \binom{n_{(2m+1)j - 2m-1}}{2} + \binom{n_{(2m+1)j-m} + 1}{2} \nonumber \\
& + \sum_{j=1}^{p} \Biggl[ \binom{n_{(2m+1)j-2m} + 1}{2} - n_{(2m+1)j-2m} n_{(2m+1)j-2m-1} + \binom{n_{(2m+1)j-m-1}}{2} \nonumber \\
& \quad \quad - n_{(2m+1)j-m} n_{(2m+1)j-m-1} + \frac{1}{2} \sum_{s=2}^{m} S(m,j,s) \Biggr] 
\end{align}

\noindent and the fact that

\begin{equation} \label{news1}
\sum_{j=1}^{p} \Biggl( \sum_{i=1}^{m} n_{(2m+1)j - 2m + i-1} - \sum_{i=m+1}^{2m} n_{(2m+1)j - 2m+i-1} \Biggr) =  \sum_{i=1}^{(2m+1)p-1} \beta_{i,m} n_i
\end{equation}

\noindent where $\beta_{i,m}$ is given by (\ref{betadef}). We now sketch how to proceed from (\ref{newlhs2}) to (\ref{Deltasum}). For $1 \leq i \leq 9$, let $\hat{L}_i$ denote the $i$th line of \eqref{newlhs2}. First, note that

\begin{equation} \label{hatl8l9}
\hat{L}_8 + \hat{L}_9 = \sum_{j=0}^{p-1} \sum_{i=1}^{m} n_{(2m+1)j + i}^2 - \sum_{j=0}^{p-1} \sum_{i=1}^{m+1} n_{(2m+1)j + i}n_{(2m+1)j + i-1}.
\end{equation}

\noindent Next, the sum over $k'$ in both $\hat{L}_1$ and $\hat{L}_2$ telescope and we obtain

\begin{equation} \label{L1hatexpanded}
\begin{aligned}
\hat{L}_1 &= \sum_{k=1}^m \sum_{j=1}^p \sum_{j'=1}^j (n_{(2m+1)j - k+1} - n_{(2m+1)j-k}) n_{(2m+1)j' + k -2m-1} \\
& \qquad \qquad - \sum_{k=1}^m \sum_{j=1}^p \sum_{j'=1}^j (n_{(2m+1)j - k+1} - n_{(2m+1)j-k}) n_{(2m+1)j'-2m-1}
\end{aligned}
\end{equation}
and
\begin{equation} \label{L2hatexpanded}
\begin{aligned}
\hat{L}_2 &= -\sum_{k=1}^m \sum_{j=1}^p \sum_{j'=1}^j (n_{(2m+1)j - m - k+1} - n_{(2m+1)j-m-k})n_{(2m+1)j' -2m-1} \\
& \qquad \qquad + \sum_{k=1}^m \sum_{j=1}^p \sum_{j'=1}^j (n_{(2m+1)j - m- k+1} - n_{(2m+1)j-m-k})n_{(2m+1)j'-3m + k-2}.
\end{aligned}
\end{equation}

\noindent Now the sum over $k$ in the second line of \eqref{L1hatexpanded} and the first line of \eqref{L2hatexpanded} both telescope and so
\begin{equation} \label{L1hatexpandedmore}
\begin{aligned}
\hat{L}_1 &= \sum_{k=1}^m \sum_{j=1}^p \sum_{j'=1}^j n_{(2m+1)j - k+1} n_{(2m+1)j' + k -2m-1} - \sum_{k=1}^m \sum_{j=1}^p \sum_{j'=1}^j n_{(2m+1)j-k} n_{(2m+1)j' + k - 2m - 1} \\
&\qquad + \sum_{j=1}^p \sum_{j'=1}^j n_{(2m+1)j - m} n_{(2m+1)j'-2m-1} - \sum_{j=1}^p \sum_{j'=1}^j n_{(2m+1)j}n_{(2m+1)j'-2m-1}
\end{aligned}
\end{equation}
and
\begin{equation} \label{L2hatexpandedmore}
\begin{aligned}
\hat{L}_2 &= \sum_{k=1}^m \sum_{j=1}^p \sum_{j'=1}^j n_{(2m+1)j - m- k+1} n_{(2m+1)j'-3m+k-2} \\
& \qquad - \sum_{k=1}^m \sum_{j=1}^p \sum_{j'=1}^j n_{(2m+1)j-m-k} n_{(2m+1)j'-3m+k-2} \\
& \qquad \qquad + \sum_{j=1}^p \sum_{j'=1}^j n_{(2m+1)j - 2m} n_{(2m+1)j' -2m-1} - \sum_{j=1}^p \sum_{j'=1}^j n_{(2m+1)j-m}n_{(2m+1)j' -2m-1}.
\end{aligned}
\end{equation}

\noindent Observe that the first sum in the second line of (\ref{L1hatexpandedmore}) cancels with the second sum in the third line of (\ref{L2hatexpandedmore}). Combine the remaining double sums, then remove the $j'=j$ term to obtain cancellation with the double sum in $\hat{L}_5$. The second sum in this $j'=j$ term then cancels with the remaining sum in $\hat{L}_5$. Next, the $i=1$ term of the second sum of (\ref{hatl8l9}) cancels with the first sum in this $j'=j$ term. Putting this together and expanding sums, we now have that (\ref{newlhs2}) equals

\begin{align} \label{newlhs3}
& \sum_{k=1}^{m} \sum_{j=1}^{p} \sum_{j'=1}^{j} n_{(2m+1)j -k+1} n_{(2m+1)j' + k -2m-1} - \sum_{k=1}^{m} \sum_{j=1}^{p} \sum_{j'=1}^{j} n_{(2m+1)j - k} n_{(2m+1)j' + k -2m-1} \nonumber \\
& + \sum_{k=1}^{m} \sum_{j=1}^{p} \sum_{j'=1}^{j} n_{(2m+1)j-m-k+1} n_{(2m+1)j'-3m+k-2} - \sum_{k=1}^{m} \sum_{j=1}^{p} \sum_{j'=1}^{j} n_{(2m+1)j - m-k} n_{(2m+1)j'-3m+k-2} \nonumber \\
& + \sum_{s=1}^{m} \sum_{j=1}^{p} \sum_{j'=1}^{j} n_{(2m+1)j'-s} n_{(2m+1)j - 2m + s-1} - \sum_{s=1}^{m} \sum_{j=1}^{p} \sum_{j'=1}^{j} n_{(2m+1)j' - 2m + s-1} n_{(2m+1)j - 2m + s-1} \nonumber \\
& + \sum_{s=1}^{m} \sum_{j=1}^{p} \sum_{j'=j+1}^{p} n_{(2m+1)j'-s} n_{(2m+1)j - 2m + s-1} - \sum_{s=1}^{m} \sum_{j=1}^{p} \sum_{j'=j+1}^{p} n_{(2m+1)j'-2m+s} n_{(2m+1)j - 2m + s-1} \nonumber \\
& - \sum_{s=1}^{m} \sum_{j=1}^{p} \sum_{j'=1}^{j-1} n_{(2m+1)j' - s} n_{(2m+1)j-s} + \sum_{s=1}^{m} \sum_{j=1}^{p} \sum_{j'=1}^{j-1} n_{(2m+1)j' - 2m + s-1} n_{(2m+1)j - s} \nonumber \\
& - \sum_{s=1}^{m} \sum_{j=1}^{p} \sum_{j'=j}^{p} n_{(2m+1)j' - s+1} n_{(2m+1)j-s} + \sum_{s=1}^{m} \sum_{j=1}^{p} \sum_{j'=j}^{p} n_{(2m+1)j' - 2m + s-1} n_{(2m+1)j-s} \nonumber \\
& + \sum_{j=0}^{p-1} \sum_{i=1}^{m} n_{(2m+1)j + i}^2 - \sum_{j=0}^{p-1} \sum_{i=2}^{m+1} n_{(2m+1)j+i} n_{(2m+1)j + i-1}.
 \end{align}

\noindent We combine the $j'=j$ term from the first sum on the third line in (\ref{newlhs3}) with the first sum in the fourth line and then cancel with the second sum in the first line. Next, the $j'=j$ term in the second sum of the third line cancels with the first sum in the last line. Thus, (\ref{newlhs3}) equals

\begin{align} \label{newlhs4}
& \sum_{k=1}^{m} \sum_{j=1}^{p} \sum_{j'=1}^{j} n_{(2m+1)j -k+1} n_{(2m+1)j' + k -2m-1} \nonumber \\
& +  \sum_{k=1}^{m} \sum_{j=1}^{p} \sum_{j'=1}^{j} n_{(2m+1)j-m-k+1} n_{(2m+1)j'-3m+k-2} - \sum_{k=1}^{m} \sum_{j=1}^{p} \sum_{j'=1}^{j} n_{(2m+1)j - m-k} n_{(2m+1)j'-3m+k-2} \nonumber \\
& + \sum_{s=1}^{m} \sum_{j=1}^{p} \sum_{j'=1}^{j-1} n_{(2m+1)j'-s} n_{(2m+1)j - 2m + s-1} - \sum_{s=1}^{m} \sum_{j=1}^{p} \sum_{j'=1}^{j-1} n_{(2m+1)j' - 2m + s-1} n_{(2m+1)j - 2m + s-1} \nonumber \\
& - \sum_{s=1}^{m} \sum_{j=1}^{p} \sum_{j'=j+1}^{p} n_{(2m+1)j'-2m+s} n_{(2m+1)j - 2m + s-1} \nonumber \\
& - \sum_{s=1}^{m} \sum_{j=1}^{p} \sum_{j'=1}^{j-1} n_{(2m+1)j' - s} n_{(2m+1)j-s} + \sum_{s=1}^{m} \sum_{j=1}^{p} \sum_{j'=1}^{j-1} n_{(2m+1)j' - 2m + s-1} n_{(2m+1)j - s} \nonumber \\
& -  \sum_{s=1}^{m} \sum_{j=1}^{p} \sum_{j'=j}^{p} n_{(2m+1)j' - s+1} n_{(2m+1)j-s} + \sum_{s=1}^{m} \sum_{j=1}^{p} \sum_{j'=j}^{p} n_{(2m+1)j' - 2m + s-1} n_{(2m+1)j-s} \nonumber \\
& - \sum_{j=0}^{p-1} \sum_{i=2}^{m+1} n_{(2m+1)j+i} n_{(2m+1)j + i-1}.
\end{align}

\noindent Now, the last line of (\ref{newlhs4}) is the $j'=j$ term of the fourth line. In the second line, perform the shift $j' \to j' +1$ and start the sum at $j'=1$. The second sum in this line then cancels with the second sum in the penultimate line, except for the $j'=j$ term. But this term now becomes the $j'=j$ term for the second sum in the fifth line. After simplifying and gathering terms, we have

\begin{align} \label{newlhs5}
& \sum_{k=1}^{m} \sum_{j=1}^{p} \sum_{j'=1}^{j} n_{(2m+1)j -k+1} n_{(2m+1)j' + k -2m-1} + \sum_{k=1}^{m} \sum_{j=1}^{p} \sum_{j'=1}^{j-1} n_{(2m+1)j - m - k + 1} n_{(2m+1)j' - m + k-1} \nonumber \\
& + \sum_{s=1}^{m} \sum_{j=1}^{p} \sum_{j'=1}^{j-1} n_{(2m+1)j'-s} n_{(2m+1)j - 2m + s-1} + \sum_{s=1}^{m} \sum_{j=1}^{p} \sum_{j'=1}^{j} n_{(2m+1)j' - 2m + s-1} n_{(2m+1)j - s} \nonumber \\
& - \sum_{s=1}^{m} \sum_{j=1}^{p} \sum_{j'=1}^{j-1} n_{(2m+1)j' - 2m + s-1} n_{(2m+1)j - 2m + s-1} - \sum_{s=1}^{m} \sum_{j=1}^{p} \sum_{j'=j}^{p} n_{(2m+1)j'-2m+s} n_{(2m+1)j - 2m + s-1} \nonumber \\
& - \sum_{s=1}^{m} \sum_{j=1}^{p} \sum_{j'=1}^{j-1} n_{(2m+1)j' - s} n_{(2m+1)j-s} - \sum_{s=1}^{m} \sum_{j=1}^{p} \sum_{j'=j}^{p} n_{(2m+1)j' - s+1} n_{(2m+1)j-s}.
\end{align}

\noindent Now we see that this is equal to (\ref{Deltasum}) as follows. The first line of (\ref{newlhs5}) corresponds to $(i,-i+1) \pmod{2m+1}$ while the second line corresponds to $(i,-i) \pmod{2m+1}$. The first sums in the third and fourth lines correspond to $(i,i) \pmod{2m+1}$ while the second sums in the same lines correspond to $(i,i+1) \pmod{2m+1}$. Thus, we have proven that (\ref{newlhs2}) equals (\ref{Deltasum}).

\end{proof}

\section{Proof of Theorem \ref{main3}}
Before proving Theorem \ref{main3}, we briefly review the theory of Bailey pairs \cite{An1,An2}. Two sequences $(\alpha_n,\beta_n)$ are said to form a {\it Bailey pair} relative to $a$ if 
\begin{equation} \label{Baileypairdef}
\beta_n = \sum_{k=0}^n \frac{\alpha_k}{(q)_{n-k}(aq)_{n+k}}.
\end{equation}
The {\it Bailey lemma} says that if $(\alpha_n,\beta_n)$ form a Bailey pair relative to $a$, then so do $(\alpha_n',\beta_n')$, where
\begin{equation} \label{alphaprimedef}
\alpha'_n = \frac{(\rho_1)_n(\rho_2)_n(aq/\rho_1 \rho_2)^n}{(aq/\rho_1)_n(aq/\rho_2)_n}\alpha_n
\end{equation} 

\noindent and

\begin{equation} \label{betaprimedef}
\beta'_n = \sum_{k=0}^n\frac{(\rho_1)_k(\rho_2)_k(aq/\rho_1 \rho_2)_{n-k} (aq/\rho_1 \rho_2)^k}{(aq/\rho_1)_n(aq/\rho_2)_n(q)_{n-k}} \beta_k.
\end{equation}
Iterating \eqref{alphaprimedef} and \eqref{betaprimedef} gives what is called the {\it Bailey chain}.

We shall not require the full power of the Bailey chain, but only two special cases. First, take the Bailey pair relative to $q$ \cite[p.468, B(3)]{Sl1},
\begin{equation} \label{alpha1}
\alpha_n = \frac{(-1)^nq^{n(3n+1)/2}(1-q^{2n+1})}{1-q}
\end{equation}
and
\begin{equation} \label{beta1}
\beta_n = \frac{1}{(q)_n}.
\end{equation}
Iterating (\ref{alpha1}) and (\ref{beta1}) using \eqref{alphaprimedef} and \eqref{betaprimedef} with $\rho_1, \rho_2 \to \infty$ at each step, we find that $(\alpha_n^{(p)},\beta_n^{(p)})$ is a Bailey pair relative to $q$, where
\begin{equation} \label{alphanp1}
\alpha_n^{(p)} = \frac{(-1)^nq^{n(n-1)/2 + p(n^2+n)}(1-q^{2n+1})}{1-q}
\end{equation}
and
\begin{equation} \label{betanp1}
\beta_n^{(p)} = \frac{1}{(q)_n} \sum_{n = n_p \geq n_{p-1} \geq \cdots \geq n_1 \geq 0} \prod_{j=1}^{p-1} q^{n_j^2+n_j} \begin{bmatrix} n_{j+1} \\ n_j \end{bmatrix}.
\end{equation}

Next take the Bailey pair relative to $q$ \cite[Eq. (4.12)]{Wa1},
\begin{equation} \label{alpha2}
\alpha_n = \frac{q^{n^2}(1-q^{2n+1})}{1-q}
\end{equation}
and
\begin{equation} \label{beta2}
\beta_n = \frac{1}{(q)_n^2}.
\end{equation}
Performing the same iteration as above to (\ref{alpha2}) and (\ref{beta2}), we find that $(\alpha_n^{(p)},\beta_n^{(p)})$ is  a Bailey pair relative to $q$, where
\begin{equation} \label{alphanp2}
\alpha_n^{(p)} = \frac{q^{pn^2+ (p-1)n}(1-q^{2n+1})}{1-q}
\end{equation}
and
\begin{equation} \label{betanp2}
\beta_n^{(p)} = \frac{1}{(q)_n} \sum_{n = n_p \geq n_{p-1} \geq \cdots \geq n_1 \geq 0} \frac{1}{(q)_{n_1}} \prod_{j=1}^{p-1} q^{n_j^2+n_j} \begin{bmatrix} n_{j+1} \\ n_j \end{bmatrix}.
\end{equation}

We are now ready to prove Theorem \ref{main3}.

\begin{proof}[Proof of Theorem \ref{main3}]
We began by recalling a formula of Walsh \cite[Cor 4.2.4, corrected]{walsh}.   Namely, for $m \geq 1$ and $p \neq 0$, we have

\begin{equation} \label{WalshcJp}
\begin{aligned}
J_N&(K_{(m,p)};a^2)  \\
&= \frac{a^{2p(1-N^2)}}{[N]} \sum_{n=0}^{N-1} \frac{[N+n]!}{[N-n-1]! [2n+1]!} c_{n,p}' \frac{(-1)^n \{2n+1\}! \{n\}!}{\{1\}(a-a^{-1})^{2n}}\sum_{k=0}^n \frac{[2k+1]}{[n+k+1]![n-k]!}\mu_{2k}^{\frac{2m-1}{2}},
\end{aligned}
\end{equation}
where
\begin{equation*}
\mu_i = a^{\frac{i^2+2i}{2}}, \quad \{n\} = a^n - a^{-n}, \quad [n] = \frac{a^n - a^{-n}}{a-a^{-1}},
\end{equation*}
\begin{equation*}
\{n\}! = \{n\}\{n-1\} \cdots \{1\}, \quad [n]! = [n][n-1] \cdots [1]
\end{equation*}
and
\begin{equation*}
c_{n,p}' = \frac{1}{(a-a^{-1})^n} \sum_{k=0}^n (-1)^k\mu_{2k}^p[2k+1]\frac{[n]!}{[n+k+1]![n-k]!}.
\end{equation*}
We note that the prefactor $a^{2p(1-N^2)}$ and the normalization factor $\frac{1}{[N]}$ are both missing in \cite{walsh}\footnote{We thank Katherine Walsh Hall for providing us with the corrected version.}.   

Some routine (but tedious) simplification shows that \eqref{WalshcJp} can be written as\begin{equation} \label{Jcpndmn}
J_N(K_{(m,p)};q) = q^{p(1-N^2)} \sum_{n \geq 0} q^n(q^{1+N})_n(q^{1-N})_n c_{p,n}(q) d_{m,n}(q),
\end{equation}
where
\begin{equation} \label{ourcpn}
c_{p,n}(q) = (q)_n\sum_{k=0}^n \frac{(-1)^kq^{\binom{k}{2} + p(k^2+k)}(1-q^{2k+1})}{(q)_{n-k}(q)_{n+k+1}}
\end{equation}
and
\begin{equation} \label{ourdmn}
d_{m,n}(q) = (q)_n\sum_{k=0}^n \frac{q^{mk^2+(m-1)k}(1-q^{2k+1})}{(q)_{n-k}(q)_{n+k+1}}.
\end{equation}
Here, we have used that $a^2=q$. Now, recalling (\ref{Baileypairdef}) and comparing \eqref{ourcpn} to \eqref{alphanp1} and \eqref{betanp1}, we have that for $p >0$,
\begin{equation} \label{cpn}
c_{p,n}(q)  = \sum_{n = n_p \geq n_{p-1} \geq \cdots \geq n_1 \geq 0} \prod_{j=1}^{p-1} q^{n_j^2+n_j} \begin{bmatrix} n_{j+1} \\ n_j \end{bmatrix}.
\end{equation}
Similarly, comparing \eqref{ourdmn} to \eqref{alphanp2} and \eqref{betanp2}, we have that for $m >0$,
\begin{equation} \label{dmn}
d_{m,n}(q) = \sum_{n = n_m \geq n_{p-1} \geq \cdots \geq n_1 \geq 0} \frac{1}{(q)_{n_1}} \prod_{j=1}^{m-1} q^{n_j^2+n_j} \begin{bmatrix} n_{j+1} \\ n_j \end{bmatrix}.
\end{equation} 
Inserting (\ref{cpn}) and (\ref{dmn}) in \eqref{Jcpndmn} gives \eqref{cyclotomicmp}.    

For the case $p < 0$, a calculation using the fact that $(1/q;1/q)_n = (q)_n(-1)^nq^{-\binom{n+1}{2}}$ shows that 
\begin{equation} \label{cpn1overq}
c_{p,n}(1/q) = (-1)^nq^{n(n+3)/2}c_{-p,n}(q).
\end{equation}
Using (\ref{cpn1overq}) together with the fact that 
\begin{equation*}
\begin{bmatrix} n \\ k \end{bmatrix} _{1/q} = q^{k^2-nk} \begin{bmatrix} n \\ k \end{bmatrix}_q
\end{equation*}
gives that for $p >0$,
\begin{equation} \label{c-pn}
c_{-p,n}(q) = (-1)^nq^{-n(n+3)/2}  \sum_{n = n_p \geq n_{p-1} \geq \cdots \geq n_1 \geq 0} 
\prod_{j=1}^{p-1} q^{-n_j - n_{j+1}n_j} \begin{bmatrix} n_{j+1} \\ n_j \end{bmatrix}.
\end{equation}
Inserting (\ref{c-pn}) and \eqref{ourdmn} in \eqref{Jcpndmn} gives \eqref{cyclotomicm-p}, which completes the proof.

\end{proof}

\section{Concluding Remarks}
Recall that Habiro \cite{habiro2} showed that for a knot $K$, the colored Jones polynomial has a \emph{cyclotomic expansion} of the form
\begin{equation} \label{cyclo}
J_N(K;q) = \sum_{n \geq 0} (q^{1+N})_n(q^{1-N})_n C_n(K;q),
\end{equation}
where the \emph{cyclotomic coefficients} $C_n(K;q)$ are Laurent polynomials independent of $N$.    The formulas in \eqref{cyclotomicmp} and \eqref{cyclotomicm-p} for $J_N(K_{(m,p)};q)$ and $J_N(K_{(m,-p});q)$ closely resemble the expansion in \eqref{cyclo}, but the coefficients are neither polynomials nor independent of $N$.    It would be highly desirable to find the correct cyclotomic expansions for these knots.    We note that this has already been done by Hikami and the first author in the case of the left-handed torus knots $K_{(1,-p)}$, where we have \cite[Prop. 3.2]{hl1}
\begin{equation}
C_n(K_{(1,-p)};q) =  q^{n+1-p} \sum_{n+1=k_p \geq k_{p-1} \geq \cdots \geq k_1 \geq 1} 
\prod_{i=1}^{t-1} q^{k_i^2}
    \begin{bmatrix}
      k_{i+1} + k_i - i   + 2\sum_{j=1}^{i-1}k_j \\
      k_{i+1} - k_i
    \end{bmatrix} .
\end{equation}

Another topic for future study would be to generalize facts about the Kontsevich-Zagier series \eqref{K-Zseries} to the generalized Kontsevich-Zagier series $\mathbb{F}_{m,p}(q)$ (and/or for $\mathfrak{F}_{m,p}(q)$) .    First, given the relation to the colored Jones polynomial in \eqref{Fatroot}, we conjecture that the $\mathbb{F}_{m,p}(q)$ are quantum modular forms.    Second, as the coefficients of $F(1-q)$ enjoy a wide variety of combinatorial interpretations (see A022493 in \cite{oeis}) and interesting congruence properties \cite{ak, akl, as, garvan, gkr, straub}, it would be of great interest to determine if the same is true for $\mathbb{F}_{m,p}(1-q)$. 

Finally, can one prove Theorems \ref{main1} and \ref{main2} using difference equations? This approach was used in \cite{hikami1, hikami2} to compute (\ref{t22p+1}).

\section*{Acknowledgements}
The authors would like to thank Paul Beirne and Katherine Walsh Hall for their helpful comments and suggestions. The second author would like to thank the Max-Planck-Institut f\"ur Mathematik for their support during the completion of this paper.


\begin{thebibliography}{999}

\bibitem{ak}
S. Ahlgren, B. Kim, \emph{Dissections of a ``strange" function}, Int. J. Number Theory \textbf{11} (2015), no. 5, 1557--1562.

\bibitem{akl}
S. Ahlgren, B. Kim, and J. Lovejoy, \emph{Dissections  of strange $q$-series}, Ann. Comb., to appear.

\bibitem{An1}
G.E. Andrews, \emph{Multiple series Rogers-Ramanujan type identities}, Pacific J. Math. {\bf 114} (1984), no. 2, 267--283.

\bibitem{An2} 
G.E. Andrews, \emph{Bailey's transform, lemma, chains and tree}, in: Special functions 2000: current perspective and future directions (Tempe, AZ), 1--22, NATO Sci. Ser. II Math. Phys. Chem., 30, Kluwer Acad. Publ., Dordrecht, 2001

\bibitem{as}
G. E. Andrews, J. Sellers, \emph{Congruences for the Fishburn numbers}, J. Number Theory \textbf{161} (2016), 298--310.

\bibitem{bhl}
K. Bringmann, K. Hikami and J. Lovejoy, \emph{The modularity of the unified WRT invariants of certain Seifert manifolds}, Adv. in Appl. Math. \textbf{46} (2011), no. 1-4, 86--93.

\bibitem{BOPR} 
J. Bryson, K. Ono, S. Pitman, and R.C. Rhoades, \emph{Unimodal sequences and quantum and mock modular forms}. Proc. Natl. Acad. Sci. USA {\bf 109} (2012), no. 40, 16063--16067.

\bibitem{bz}
G. Burde, H. Zieschang, \emph{Knots}, De Gruyter Studies in Mathematics, \textbf{5}. De Gruyter, Berlin, 2014.

\bibitem{Co1}
H. Cohen, \emph{$q$-identities for Maass waveforms}, Invent. Math. {\bf 91} (1988), no. 3, 409--422.

\bibitem{folsometal}
A. Folsom, C. Ki, Y.N. Truong Vu, and B. Yang, \emph{Strange combinatorial quantum modular forms}, J. Number Theory {\bf 170} (2017), 315--346.

\bibitem{garvan}
F. G. Garvan, \emph{Congruences and relations for $r$-Fishburn numbers}, J. Combin. Theory Ser. A \textbf{134} (2015), 147--165.

\bibitem{gkr}
P. Guerzhoy, Z. Kent and L. Rolen, \emph{Congruences for Taylor expansions of quantum modular forms}, Res. Math. Sci. \textbf{1} (2014), Art. 17, 17pp.

\bibitem{gukov}
S. Gukov, \emph{Three-dimensional quantum gravity, Chern-Simons theory, and the $A$-polynomial}, Comm. Math. Phys. \textbf{255} (2005), no. 3, 577--627.

\bibitem{habiro}
K. Habiro, \emph{On the colored Jones polynomial of some simple links}, in: Recent progress toward the volume conjecture (Kyoto, 2000), S{\=u}rikaisekikenky{\=u}sho K{\=o}ky{\=u}roku \textbf{1172} (2000), 34--43.

\bibitem{habiro2}
K. Habiro, \emph{A unified Witten-Reshetikhin-Turaev invariant for homology spheres}, Invent. Math. \textbf{171} (2008), no. 1, 1--81.

\bibitem{hikami1}
K. Hikami, \emph{Difference equation of the colored Jones polynomial for torus knot}, Internat. J. Math. \textbf{15} (2004), no. 9, 959--965.

\bibitem{hikami2}
K. Hikami, \emph{$q$-series and $L$-functions related to half-derivates of the Andrews-Gordon identity}, Ramanujan J. \textbf{11} (2006), no. 2, 175--197.

\bibitem{hikami3}
K. Hikami, \emph{Asymptotics of the colored Jones polynomial and the $A$-polynomial}, Nuclear Phys. B \textbf{773} (2007), no. 3, 184--202.

\bibitem{hikami4}
K. Hikami, \emph{Hecke type formula for unified Witten-Reshetikhin-Turaev invariants as higher-order mock theta functions}, Int. Math. Res. Not. IMRN 2007, no. 7, Art. ID rnm 022, 32pp.

\bibitem{hl1}
K. Hikami, J. Lovejoy, \emph{Torus knots and quantum modular forms}, Res. Math. Sci. \textbf{2} (2015), Art. 2, 15pp.  
  
\bibitem{hl2}
K. Hikami, J. Lovejoy, \emph{Hecke-type formulas for families of unified Witten-Reshetikhin-Turaev invariants}, Commun. Number Theory Phys. \textbf{11} (2017), no. 2, 249--272. 

\bibitem{thang}
T. T. Q. L{\^e}, \emph{Quantum invariants of 3-manifolds: Integrality, splitting, and perturbative expansion}, Topology Appl. \textbf{127} (2003), no. 1-2, 125--152.

\bibitem{La}  
M. R. Lauridsen, \emph{Aspects of quantum mathematics, Hitchin connections and AJ conjectures}, Ph.D. thesis, Aarhus University, Aarhus, Denmark, 2010. 
  
\bibitem{loCJP}
J. Lovejoy, R. Osburn, \emph{The colored Jones polynomial and Kontsevich-Zagier series for double twist knots}, preprint available at \url{https://arxiv.org/abs/1710.04865}  
 
\bibitem{mpvl} 
M. Macasieb, K. Petersen and R. van Luijk, \emph{On character varieties of two-bridge knot groups}, Proc. Lond. Math. Soc. (3) \textbf{103} (2011), no. 3, 473--507.  
  
\bibitem{masbaum}
G. Masbaum, \emph{Skein-theoretical derivation of some formulas of Habiro}, Algebr. Geom. Topol. \textbf{3} (2003), 537--556.

\bibitem{Sl1}
L.J. Slater, \emph{A new proof of Rogers's transformations of infinite series}, Proc. London Math. Soc. (2) {\bf 53} (1951), 460--475.

\bibitem{oeis}
N. J. A. Sloane, \emph{On-Line Encyclopedia of Integer Sequences}, available at \url{http://oeis.org}

\bibitem{straub}
A. Straub, \emph{Congruences for Fishburn numbers modulo prime powers}, Int. J. Number Theory \textbf{11} (2015), no. 5, 1679--1690.

\bibitem{takata}
T. Takata, \emph{A formula for the colored Jones polynomial of 2-bridge knots}, Kyungpook Math. J. \textbf{48} (2008), no. 2, 255--280.

\bibitem{Tran}
A. Tran, \emph{Nonabelian representations and signatures of double twist knots}, J. Knot Theory Ramifications \textbf{25} (2016), no. 3, 1640013, 9pp.

\bibitem{walsh}
K. Walsh, \emph{Patterns and stability in the coefficients of the colored Jones polynomial}, Ph.D. thesis, University of California, San Diego, 2014.

\bibitem{Wa1}
S.O. Warnaar, \emph{Partial theta functions. I. Beyond the lost notebook}, Proc. London Math. Soc. (3) {\bf 87} (2003), no. 2, 363--395.

\bibitem{z-1}
D. Zagier, \emph{Vassiliev invariants and a strange identity related to the Dedekind eta-function}, Topology {\bf 40} (2001), no. 5, 945--960.

\bibitem{zagier}
D. Zagier, \emph{Quantum modular forms}, Quanta of maths, 659--675, Clay Math. Proc., \textbf{11}, Amer. Math. Soc., Providence, RI, 2010.

\end{thebibliography}
\end{document}